\documentclass{amsart}
\numberwithin{equation}{section}
\usepackage[english]{babel}
\usepackage[T1]{fontenc}
\usepackage[latin1]{inputenc}
\usepackage{indentfirst}
\usepackage{amsmath,amssymb, amsbsy}
\usepackage{amsfonts}
\usepackage{latexsym}
\usepackage{amsthm}
\usepackage[dvips]{graphicx}
\DeclareGraphicsExtensions{.pdf,.png,.jpg,.eps}
\usepackage[usenames,dvipsnames]{color}

\usepackage{newtxtext,newtxmath,helvet,courier}
\usepackage{fix-cm}

\usepackage{tikz}

\usepackage{color}
\usepackage[latin1]{inputenc}
\usepackage[active]{srcltx}
\definecolor{navyblue}{RGB}{0, 25, 111}

\newcommand{\brd}[1]{\mathbb{#1}}
\newcommand{\R}{\brd{R}}
\newcommand{\abs}[1]{\left\lvert {#1} \right\rvert}
\newcommand{\norm}[2]{\left\Vert {#1} \right\Vert_{#2}}
\newcommand\ddfrac[2]{\frac{\displaystyle #1}{\displaystyle #2}}
\newtheorem{teo}{Theorem}[section]
\newtheorem{Corollary}[teo]{Corollary}
\newtheorem{Lemma}[teo]{Lemma}
\newtheorem{Theorem}[teo]{Theorem}
\newtheorem{Proposition}[teo]{Proposition}
\theoremstyle{definition}
\newtheorem{Definition}[teo]{Definition}

\newtheorem{remark}[teo]{Remark}
\newtheorem{problem}[teo]{Problem}

\begin{document}

\title[On $s$-harmonic functions on cones]
{On S-harmonic functions on cones}
\author{Susanna Terracini}
\author{Giorgio Tortone}
\author{Stefano Vita}

\date{\today}

\address{Susanna Terracini, Giorgio Tortone and Stefano Vita \newline \indent
 Dipartimento di Matematica ``Giuseppe Peano'', \newline \indent
Universit\`a di Torino, \newline \indent
Via Carlo Alberto, 10,
10123 Torino, Italy}
\email{susanna.terracini@unito.it}
\email{giorgio.tortone@unito.it}
\email{stefano.vita@unito.it}

\keywords{Fractional Laplacian, conic functions, asymptotic behaviour, Martin kernel}
\subjclass{35R11 (35B45,35B08)}

\thanks{{\bf Acknowlegments.}   Work partially supported by the
ERC Advanced Grant 2013 n. 339958
{\it Complex Patterns for Strongly Interacting Dynamical Systems - COMPAT} and by the INDAM-GNAMPA project {\it Aspetti non-locali in fenomeni di segregazione}. The authors wish to thank Alessandro Zilio for many fruitful conversations.
}

\maketitle

\begin{abstract}
We deal with non negative functions satisfying

\begin{equation*}\label{P_C}
\begin{cases}
(-\Delta)^s u_s=0 & \mathrm{in}\quad C, \\
u_s=0 & \mathrm{in}\quad \mathbb{R}^n\setminus C,
\end{cases}
\end{equation*}

where $s\in(0,1)$ and $C$ is a given cone on $\mathbb R^n$ with vertex at zero. We consider the case when $s$ approaches $1$,  wondering whether solutions of the problem do converge to harmonic functions in the same cone or not. Surprisingly, the answer will depend on the opening of the cone through an auxiliary eigenvalue problem on the upper half sphere.
These conic functions are involved in the study of the nodal regions in the case of optimal partitions and other free boundary problems and play a crucial role in the extension of the Alt-Caffarelli-Friedman monotonicity formula to the case of fractional diffusions.

\end{abstract}

\tableofcontents

\section{Introduction}
Let $n\geq 2$ and let $C$ be an open cone in $\mathbb{R}^n$ with vertex in $0$;
for a given $s\in(0,1)$, we consider the problem of the classification of nontrivial functions which are $s$-harmonic inside the cone and vanish identically outside, that is:
\begin{equation}\label{Pstheta}
\begin{cases}
(-\Delta)^s u_s=0 & \mathrm{in}\quad C, \\
u_s\geq 0 & \mathrm{in}\quad \mathbb{R}^n\\  u_s\equiv0 & \mathrm{in}\quad \mathbb{R}^n\setminus C.
\end{cases}
\end{equation}
Here we define (see \S\ref{sec:2} for the details)
$$
(-\Delta)^s u(x)=  C(n,s) \mbox{ P.V.}\int_{\R^n}{\frac{u(x)-u(\eta)}{\abs{x-\eta}^{n+2s}}\mathrm{d}\eta}\;,
$$
where  $u$ is a sufficiently smooth function and
\begin{equation}\label{eq:Cns}
C(n,s) = \frac{2^{2s}s\Gamma(\frac{n}{2}+s)}{\pi^{n/2}\Gamma(1-s)}>0,
\end{equation}
where
$$
\Gamma(x)=\int_{0}^\infty{t^{x-1}e^{-t}\mathrm{d}t}.
$$
The principal value is taken at $\eta=x$: hence, though $u$ needs not to decay at infinity, it has to keep an algebraic growth with a power strictly smaller than $2s$ in order to make the above expression meaningful. By Theorem 3.2 in \cite{banuelos}, it is known that there exists a homogeneous, nonnegative and nontrivial solution to \eqref{Pstheta} of the form
\begin{equation*}
u_s(x)=|x|^{\gamma_s}u_s\left(\frac{x}{|x|}\right),
\end{equation*}
where $\gamma_s:=\gamma_s(C)$ is a definite homogeneity degree (characteristic exponent), which depends on the cone. Moreover, such a solution is continuous in $\mathbb{R}^n$ and unique, up to multiplicative constants. We can normalize it in such a way that  $\norm{u_s}{L^\infty(S^{n-1})}=1$. We consider the case when $s$ approaches $1$,  wondering whether solutions of the problem do converge to a harmonic function in the same cone and, in case, which are the suitable spaces for convergence.

\bigskip

Such conic $s$-harmonic functions appear as limiting blow-up profiles and play a major role in many free boundary problems with fractional diffusions and in the study of the geometry of nodal sets, also in the case of partition problems (see, e.g. \cite{MR3039830,BaFiRo, CafSilvaSavin, allfunctions,GaRo}). Moreover, as we shall see later, they are strongly involved with the possible extensions of the Alt-Caffarelli-Friedman monotonicity formula to the case of fractional diffusion. The study of their properties and, ultimately,  their classification is therefore a major achievement in this setting.  The problem of homogeneous $s$-harmonic functions on cones has been deeply studied in \cite{banuelos,MR1671973, Bogdan.narrow,  MR2213639}. The present paper mainly focuses on the limiting behaviour as $s\nearrow 1$.

\bigskip
Our problem \eqref{Pstheta} can be linked to a specific spectral problem of local nature in the upper half sphere; indeed let us look at the extension technique popularized by Caffarelli and Silvestre (see \cite{MR2354493}), characterizing the fractional Laplacian in $\mathbb{R}^n$ as the Dirichlet-to-Neumann map for a variable $v$ depending on one more space dimension and satisfying:
\begin{equation}\label{eq:Pextended}
\begin{cases}
L_s v=\mathrm{div}(y^{1-2s}\nabla v)=0\qquad &\mathrm{in}\quad\mathbb{R}^{n+1}_+,\\
v(x,0)=u(x) & \mathrm{on}\quad\mathbb{R}^{n}\;.
\end{cases}
\end{equation}
Such an extension exists unique for a suitable class of functions $u$ (see \eqref{norml1s}) and it is given by the formula:
\begin{equation*}
v(x,y)=\gamma(n,s)\int_{\R^n} \frac{y^{2s}u(\eta)}{(|x-\eta|^2+y^2)^{n/2+s}} \mathrm{d}\eta \qquad \mathrm{where}\; \gamma(n,s)^{-1}:=\int_{\R^n} \frac{1}{(|\eta|^2+1)^{n/2+s}} \mathrm{d}\eta\;.
\end{equation*}
Then, the nonlocal original operator translates into a boundary derivative operator of Neumann type:
$$
-\dfrac{C(n,s)}{\gamma(n,s)}\lim_{y\to0}y^{1-2s}\partial_yv(x,y)=(-\Delta)^su(x).$$
Now, let us consider an open region $\omega\subseteq S^{n-1}=\partial S^n_+$, with $S^{n}_+= S^{n}\cap \{ y>0\}$, and define the eigenvalue
$$
\lambda_1^s(\omega)=\inf\left\{\ddfrac{\int_{S^n_+}y^{1-2s}|\nabla_{S^{n}}u|^2\mathrm{d}\sigma}{\int_{S^n_+}y^{1-2s}u^2\mathrm{d}\sigma} \ : \ u\in H^1(S^n_+;y^{1-2s}\mathrm{d}\sigma)\setminus\{0\} \ \mathrm{and} \ u\equiv0 \ \mathrm{in} \ S^{n-1}\setminus\omega\right\}.
$$
Next, define the \emph{characteristic exponent} of the cone $C_\omega$ spanned by $\omega$ (see Definition \ref{spanned}) as
\begin{equation}\label{eq:characteristic}
\gamma_s(C_\omega)=\gamma_s(\lambda_1^s(\omega))\;,
\end{equation}
where the function $\gamma_s(t)$ is defined by
$$
\gamma_s(t):=\sqrt{\bigg(\frac{n-2s}{2}\bigg)^2 +t}-\frac{n-2s}{2}\;.$$

\begin{remark}
There is a remarkable link between the nonnegative $\lambda_1^s(\omega)$-eigenfunctions  and the $\gamma_s(\lambda_1^s(\omega))$-homogeneous $L_s$-harmonic functions: let
consider the spherical coordinates $(r,\theta)$ with $r>0$ and $\theta\in S^{n}$. Let $\varphi_s$ be the first nonnegative eigenfunction to $\lambda_1^s(\omega)$ and let $v_s$ be its $\gamma_s(\lambda_1^s(\omega))$-homogeneous extension to $\mathbb{R}^{n+1}_+$, i.e.
$$
v_s(r,\theta)= r^{\gamma_s(\lambda_1^s(\omega))}\varphi_s(\theta),
$$
which is well defined as soon as
\(\gamma_s(\lambda_1^s(\omega))<2s\) (as we shall see, this fact is always granted). By \cite{spectralLs}, the operator $L_s$ can be decomposed as
$$
L_s u=\sin^{1-2s}(\theta_n) \frac{1}{r^n}\partial_r \left(r^{n+1+2s}\partial_r u\right) + \frac{1}{r^{1+2s}} L_s^{S^n} u
$$
where $y=r \sin(\theta_n)$ and the Laplace-Beltrami type operator is defined as
$$
L_s^{S^n} u = \mbox{div}_{S^n}(\sin^{1-2s}(\theta_n) \nabla_{S^n} u)
$$
with $\nabla_{S^n}$ the tangential gradient on $S^n$. Then, we easily get that $v_s$ is  $L_s$-harmonic in the upper half-space; moreover its trace $u_s(x)=v_s(x,0)$ is $s$-harmonic in the cone $C_\omega$ spanned by $\omega$, vanishing identically outside: in other words $u_s$ is a solution of our problem \eqref{Pstheta}.
\end{remark}

In a symmetric way, for the standard Laplacian, we consider the problem of $\gamma$-homogeneous functions which are harmonic inside the cone spanned by $\omega$ and vanish outside:
\begin{equation}\label{P1theta}
\begin{cases}
-\Delta u_1=0 & \mathrm{in}\quad C_\omega, \\
u_1\geq0 & \mathrm{in}\quad \mathbb{R}^n\\
u_1=0 & \mathrm{in}\quad \mathbb{R}^n\setminus C_\omega.
\end{cases}
\end{equation}
Is is well known that the associated eigenvalue problem on the sphere is that of the Laplace-Beltrami operator with Dirichlet boundary conditions:

$$
\lambda_1(\omega)=\inf\left\{\ddfrac{\int_{S^{n-1}}|\nabla_{S^{{n-1}}}u|^2\mathrm{d}\sigma}{\int_{S^{n-1}}u^2\mathrm{d}\sigma} \ : \ u\in H^1(S^{n-1})\setminus\{0\} \ \mathrm{and} \ u=0 \ \mathrm{in} \ S^{n-1}\setminus\omega\right\},
$$
and the \emph{characteristic exponent} of the cone $C_\omega$ is
\begin{equation}\label{eq:1characteristic}
\gamma(C_\omega)=\sqrt{\bigg(\frac{n-2}{2}\bigg)^2 +\lambda_1(\omega)}-\frac{n-2}{2}={\gamma_s}_{\vert s=1}(\lambda_1(\omega))\;.
\end{equation}
In the classical case, the characteristic exponent enjoys a number of nice properties: it is minimal on spherical caps among sets having a given measure. Moreover for the spherical caps, the eigenvalues enjoy  a fundamental convexity property with respect to the colatitude $\theta$ (\cite{MR732100,MR0412442}).  The convexity plays a major role in the proof of the Alt-Caffarelli-Friedman monotonicity formula, a key tool in the Free Boundary Theory (\cite{MR2145284}).

\bigskip

Since the standard Laplacian can be viewed as the limiting operator of the family $(-\Delta)^s$ as $s\nearrow 1$, some questions naturally arise:
\begin{problem} Is it true that
\begin{itemize}
\item[(a)] \( \lim_{s\to 1}\gamma_s(C)=\gamma(C)\)?
\item[(b)] \( \lim_{s\to 1}u_s=u_1\) uniformly on compact sets, or better, in H\"older local norms?
\item[(c)] for spherical caps of opening $\theta$ is there any convexity of the map $\theta\mapsto \lambda_1^s(\theta)$ at least, for $s$ near $1$?
\end{itemize}
\end{problem}

We therefore addressed the problem of the asymptotic behavior of the solutions of problem \eqref{Pstheta} for $s\nearrow 1$, obtaining a rather unexpected result: our analysis shows high sensitivity to the opening solid angle $\omega$ of the cone $C_\omega$, as evaluated by the value of $\gamma(C)$.  In the case of wide cones, when $\gamma(C)<2$ (that is, $\theta\in(\pi/4,\pi)$ for spherical caps of colatitude $\theta$), our solutions do converge to the harmonic homogeneous function of the cone; instead, in the case of narrow cones, when $\gamma(C)\geq2$ (that is, $\theta\in(0,\pi/4]$ for spherical caps),
then limit of the homogeneity degree will be always two and the limiting profile will be something different, though related, of course, through a correction term.  Similar transition phenomena have been detected in other contexts for some types of free boundary problems on cones (\cite{AlCh2015, Sha2004}). As a consequence of our main result, we will see a lack of convexity of the eigenvalue as a function of the colatitude. Our main result is the following Theorem.
\begin{Theorem}\label{teolimit1}
Let $C$ be an open cone with vertex at the origin. There exist finite the following limits:

\begin{equation*}\label{limit.mu1}
\overline\gamma(C):=\lim_{s\to 1^-}\gamma_s(C)=\min\{\gamma(C),2\}
\end{equation*}
and
\begin{equation*}\label{limit.mu2}
\mu(C):=\lim_{s\to 1^-}\frac{C(n,s)}{2s-\gamma_s(C)}=\begin{cases}
0 & \mathrm{if} \  \gamma(C)\leq2,\\
\mu_0(C) & \mathrm{if} \  \gamma(C)\geq2,
\end{cases}
\end{equation*}
where $C(n,s)$ is defined in \eqref{eq:Cns} and
\begin{equation*}\label{mu012}
\mu_0(C):=\inf\left\{\ddfrac{\int_{S^{n-1}}\abs{\nabla_{S^{n-1}}u}^2-2nu^2\mathrm{d}\sigma}{\left(\int_{S^{n-1}}|u|\mathrm{d}\sigma\right)^2} \ : \ u\in H^1(S^{n-1})\setminus\{0\} \ \mathrm{and} \ u=0 \ \mathrm{in} \ S^{n-1}\setminus C\right\}.
\end{equation*}
Let us consider the family $(u_{s})$ of nonnegative solutions to \eqref{Pstheta} such that  $||u_{s}||_{L^\infty(S^{n-1})}=1$. Then, as $s\nearrow 1$, up to a subsequence, we have
\begin{itemize}
\item[1.] $u_{s}\to\overline u$ in $L^2_{\mathrm{loc}}(\mathbb{R}^n)$ to some $\overline u\in H^1_{\mathrm{loc}}(\mathbb{R}^n)\cap L^\infty(S^{n-1})$.
\item[2.] The convergence is uniform on compact subsets of $C$ , $\overline u$ is nontrivial with $||\overline u||_{L^\infty(S^{n-1})}=1$ and is $\overline\gamma(C)$-homogeneous.
\item[3.] The limit $\overline u$ solves
\begin{equation}
\begin{cases}\label{limit.equation1}
-\Delta\overline u=\mu(C)\displaystyle\int_{S^{n-1}}{\!\!\!\!\overline u \mathrm{d}\sigma}& \mathrm{in} \ C,\\
\overline u=0 & \mathrm{in} \ \mathbb{R}^n\setminus C\;.
\end{cases}
\end{equation}

\end{itemize}
\end{Theorem}

\begin{figure}
\begin{tikzpicture}

          \draw[->] (0,0) -- (6.5,0) node[right] {$\theta$};
          \draw[->] (0,0) -- (0,5) node[above] {$y$};
          \draw	(0,0) node[anchor=north] {0}
		(pi/2,0) node[anchor=north] {$\pi/4$}
		(pi,0) node[anchor=north] {$\pi/2$}
		(3*pi/2,0) node[anchor=north] {$3\pi/4$}
		(2*pi,0) node[anchor=north] {$\pi$}
		(0,2.2) node[anchor=east] {$2$}
		(0,1.8) node[anchor=east] {$2s$}
		(0,0.85) node[anchor=east] {$s$};
		
	\draw(2.7,2.2) node{$y=\gamma(\lambda_1(\theta))$}
		(1.1,1.2) node{$y=\gamma_s(\lambda_1^s(\theta))$};

          \draw [domain=0:pi/2,thick, dashed] plot (\x,2);
          \draw [domain=pi/2:2*pi,red,thick] plot (\x,pi/\x);
          \draw [domain=0.65:pi/2,red,thick] plot (\x,pi/\x);
          \draw [domain=0:pi,dashed] plot (\x,0.85);

          \draw [cyan,thick] plot [smooth, tension=1] coordinates { (0,1.8) (pi/2,1.6) (pi,0.85) (2*pi,0.35)};
         \end{tikzpicture}

             \caption{Characteristic exponents of spherical caps of aperture $2\theta$ for $s<1$ and $s=1$.}

         \end{figure}
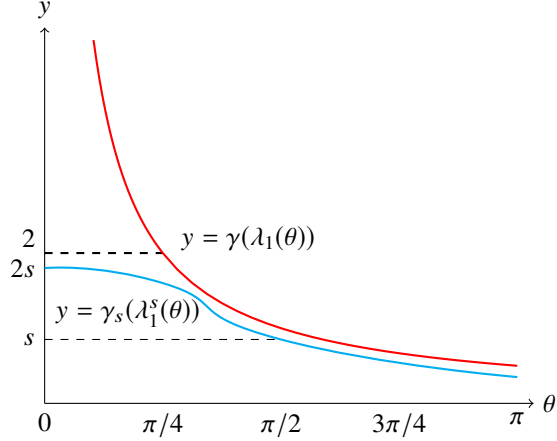
\begin{remark}
 Uniqueness of the limit $\overline u$ and therefore existence of the limit of $u_s$ as $s\nearrow 1$ holds in the case of connected cones and, in any case,  whenever $\gamma(C)>2$. We will see in Remark \ref{rem:uniqueness} that under symmetry assumptions on the cone $C$, the limit function $\overline{u}$ is unique and hence it does not depend on the choice of the subsequence. \\
  \end{remark}
    A nontrivial improvement of the main Theorem concerns uniform bounds in H\"older spaces holding uniformly for $s\to 1$.
\begin{Theorem}\label{thm:holder}
Assume the cone is $\mathcal C^{1,1}$. Let $\alpha\in(0,1)$, $s_0\in(\max\{1/2,\alpha\},1)$ and $A$ an annulus centered at zero. Then the family of solutions $u_s$ to \eqref{Pstheta} is uniformly bounded in $C^{0,\alpha}(A)$ for any $s\in[s_0,1)$.
\end{Theorem}

\subsection{On the fractional Alt-Caffarelli-Friedman monotonicity formula}

In the case of reaction-diffusion systems with strong competition  between a number of densities which spread in space, one can observe a segregation phenomenon: as the interspecific competition rate grows, the populations tend to separate their supports in nodal sets, separated by a free boundary.  For the case of standard diffusion, both the asymptotic analysis and the properties of the segregated limiting profiles are fairly well understood, we refer to \cite{MR2393430, MR2146353,MR2863857, MR2599456,MR2984134} and references therein. Instead, when the diffusion is nonlocal and modeled by the fractional Laplacian, the only known results are contained in \cite{tvz2,tvz1,2016arXiv160403264T,ww}. As shown in \cite{tvz2,tvz1}, estimates in H\"older spaces can be obtained by the use of fractional versions of the Alt-Caffarelli-Friedman (ACF) and Almgren monotonicity formul\ae. For the statement, proof and applications of the original ACF monotonicity formula we refer to the book by Caffarelli and Salsa \cite{MR2145284} on free boundary problems. Let us state here the fractional version of the spectral problem beyond the ACF formula used in \cite{tvz2,tvz1}: consider the set of $2$-partitions of $S^{n-1}$ as
$$
\mathcal{P}^2:= \left\{(\omega_1,\omega_2): \omega_i\subseteq S^{n-1} \mbox{ open, }\omega_1\cap\omega_2 = \emptyset, \ \overline{\omega_1}\cup\overline{\omega_2}=S^{n-1}\right \}
$$
and define the  optimal partition value as:
\begin{equation}
\nu^{\textit{ACF}}_s := \frac{1}{2}\inf_{(\omega_1,\omega_2)\in\mathcal{P}^2}\sum_{i=1}^2 \gamma_s(\lambda_1^s(\omega_i)).
\end{equation}
It is easy to see, by a Schwarz symmetrization argument, that $\nu_s^{\textit{ACF}}$ is achieved by a pair of complementary spherical caps $(\omega_\theta,\omega_{\pi-\theta})\in\mathcal{P}^2$ with aperture $2\theta$ and $\theta\in(0,\pi)$ (for a detailed proof of this kind of symmetrization we refer to \cite{2016arXiv160403264T}), that is:
\begin{equation*}
\nu_s^{\textit{ACF}}=\min_{\theta\in[0,\pi]}\Gamma^s(\theta)=\min_{\theta\in[0,\pi]}\frac{\gamma_s(\theta)+\gamma_s(\pi-\theta)}{2}.
\end{equation*}
This gives a further motivation to our study of \eqref{Pstheta} for spherical caps. A classical result by Friedland and Hayman, \cite{MR0412442}, yields $ \nu^{\textit{ACF}}=1$ (case $s=1$), and the minimal value is achieved for two half spheres;  this equality is the core of the proof of the classical Alt-Caffarelli-Friedman monotonicity formula.

It was proved in \cite{tvz2} that $ \nu^{\textit{ACF}}_s$ is linked to the threshold for uniform bounds in H\"older norms for competition-diffusion systems, as the interspecific competition rate diverges to infinity, as well as the exponent of the optimal H\"older regularity for their limiting profiles. It was also conjectured that $ \nu^{\textit{ACF}}_s=s$ for every $s\in(0,1)$.
Unfortunately, the exact value of $\nu^{\textit{ACF}}_s$ is still unknown, and we only know that $0<\nu_s^{\textit{ACF}}\leq s$ (see \cite{tvz2,tvz1}). Actually one can easily give a better lower bound given by $\nu^{\textit{ACF}}_s\geq\max\{s/2,s-1/4\}$ when $n=2$ and $\nu^{\textit{ACF}}_s\geq s/2$ otherwise, which however it is not satisfactory.  As already remarked in \cite{MR3039830}, this lack of information implies also the lack of an exact Alt-Caffarelli-Friedman monotonicity formula for the case of fractional Laplacians.  Our contribution to this open problem is  a byproduct of the main Theorem \ref{limit.mu1}.
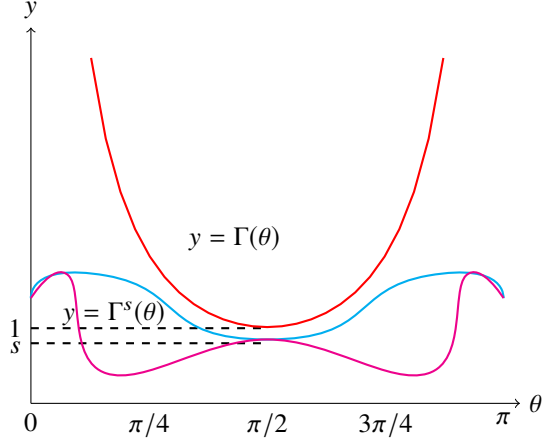
\begin{figure}
\begin{tikzpicture}

          \draw[->] (0,0) -- (6.5,0) node[right] {$\theta$};
          \draw[->] (0,0) -- (0,5) node[above] {$y$};
          \draw	(0,0) node[anchor=north] {0}
		(pi/2,0) node[anchor=north] {$\pi/4$}
		(pi,0) node[anchor=north] {$\pi/2$}
		(3*pi/2,0) node[anchor=north] {$3\pi/4$}
		(2*pi,0) node[anchor=north] {$\pi$}
		(0,1) node[anchor=east] {$1$}
		(0,0.75) node[anchor=east] {$s$};
		
	\draw(2.7,2.2) node{$y=\Gamma(\theta)$}
		(1.1,1.2) node{$y=\Gamma^s(\theta)$};

          \draw [domain=0:pi,thick, dashed] plot (\x,1);
          \draw [domain=0:pi,thick,dashed] plot (\x,0.8);
          \draw [domain=0.8:2*pi-0.8,red,thick] plot (\x,{-1.85+4.5/\x+4.5/(2*pi-\x)});
          \draw [cyan,thick] plot [smooth, tension=1.2] coordinates { (0,1.4) (1,1.7) (pi,0.85) (2*pi-1,1.7)  (2*pi,1.4)};
          \draw [magenta,thick] plot [smooth, tension=0.7] coordinates { (0,1.4) (0.5,1.7)(1,0.4) (pi,0.85) (2*pi-1,0.4)(2*pi-0.5,1.7)  (2*pi,1.4)};

         \end{tikzpicture}

             \caption{Possible values of $\Gamma^s(\theta)=\Gamma^s(\omega_\theta,\omega_{\pi-\theta})$ for $s<1$ and $s=1$ and $n=2$.}

         \end{figure}

\begin{Corollary}\label{cor1} In any space dimension we have
\[
\lim_{s\to 1}\nu^{\textit{ACF}}_s=1\;.
\]
\end{Corollary}
\bigskip
The paper is organized as follows. In Section 2 we introduce our setting and we state the relevant known properties of homogeneous $s$-harmonic functions on cones. After this, we will obtain local $C^{0,\alpha}$-estimates in compact subsets of $C$ and local $H^s$-estimates in compact subsets of $\mathbb{R}^n$ for solutions $u_s$ of \eqref{Pstheta}. We will see that an important quantity which appears in this estimates and plays a fundamental role is
$$
\frac{C(n,s)}{2s-\gamma_s(C)},
$$
where $C(n,s)>0$ is the normalization constant given in \eqref{eq:Cns}. It will be therefore very important to bound this quantity uniformly in $s$.
In Section 3 we analyze the asymptotic behaviour of $\gamma_s(C)$ as $s$ converges to $1$, in order to understand the quantities $\overline\gamma(C)$ and $\mu(C)$. To do this, we will establish a distributional semigroup property for the fractional Laplacian for functions which grow at infinity. In Section 4 we prove Theorem \ref{teolimit1} and Corollary \ref{cor1}. Eventually, in Section 5, we prove Theorem \ref{thm:holder}.

\section{Homogenous $s$-harmonic functions on cones}\label{sec:2}
In this section, we focus our attention on the local properties of homogeneous $s$-harmonic functions on \textit{regular cones}. Since in the next section we will study the behaviour of the characteristic exponent as $s$ approaches $1$, in this section we recall some known results related to the boundary behaviour of the solution of \eqref{Pstheta} restricted to the unitary sphere $S^{n-1}$ and some estimates of the H$\ddot{\mbox{o}}$lder and $H^s$ seminorm.
\begin{Definition}\label{spanned}
Let $\omega\subset S^{n-1}$ be an open set, that may be disconnected. We call \textit{unbounded cone} with vertex in $0$, spanned by $\omega$ the open set
\begin{equation*}
C_\omega=\{rx \ : \ r>0, \ x\in\omega\}.
\end{equation*}
Moreover we say that $C=C_\omega$ is narrow if $\gamma(C)\geq2$ and wide if $\gamma(C)<2$.
We call $C_\omega$ \textit{regular cone} if $\omega$ is connected and of class $C^{1,1}$. Let $\theta\in(0,\pi)$ and $\omega_\theta\subset S^{n-1}$ be an open spherical cap of colatitude $\theta$. Then we denote by $C_\theta=C_{\omega_\theta}$ the \textit{right circular cone} of aperture $2\theta$.
\end{Definition}
Hence, let $C$ be a fixed unbounded open cone in $\R^n$ with vertex in $0$ and consider
\begin{equation*}\label{Pstheta1}
\begin{cases}
(-\Delta)^s u_s=0 & \mathrm{in}\quad C, \\
u_s\equiv 0 & \mathrm{in}\quad \mathbb{R}^n\setminus C.
\end{cases}
\end{equation*}
with the condition $\norm{u_s}{L^\infty(S^{n-1})}=1$. By Theorem 3.2 in \cite{banuelos} there exists, up to a multiplicative constant, a unique nonnegative function $u_s$ smooth in $C$ and $\gamma_s(C)$-homogenous, i.e.
$$
u_s(x) = \abs{x}^{\gamma_s(C)}u_s\left(\frac{x}{\abs{x}}\right)
$$
where $\gamma_s(C) \in (0,2s)$.
As it is well know (see for example \cite{MR1671973,silvestrethesis}), the fractional Laplacian $(-\Delta)^s$ is a nonlocal operator well defined in the class of integrability $\mathcal{L}^1_s:=\mathcal{L}^1 \left({\mathrm{d}x}/{(1+|x|)^{n+2s}}\right)$, namely the normed space of all Borel functions $u$ satisfying
\begin{equation}\label{norml1s}
\norm{u}{\mathcal{L}^1_s}:=\int_{\R^n}{\frac{\abs{u(x)}}{(1+\abs{x})^{n+2s}} \mathrm{d}x}< +\infty.
\end{equation}
Hence, for every $u\in \mathcal{L}^1_s, \varepsilon>0$ and $x \in\R^n$ we define
$$
(-\Delta)^s_\varepsilon u(x) = C(n,s) \int_{\mathbb{R}^n\setminus B_\varepsilon(x)}{\frac{u(x)-u(y)}{\abs{x-y}^{n+2s}}\mathrm{d}y},
$$
where
$$
C(n,s) = \frac{2^{2s}s\Gamma(\frac{n}{2}+s)}{\pi^{n/2}\Gamma(1-s)} \in \left(0, 4\Gamma\left(\frac{n}{2}+1\right) \right].
$$
and we can consider the fractional Laplacian as the limit
$$
(-\Delta)^s u(x)= \lim_{\varepsilon \downarrow 0}(-\Delta)^s_\varepsilon u(x) = C(n,s) \mbox{ P.V.}\int_{\R^n}{\frac{u(x)-u(y)}{\abs{x-y}^{n+2s}}\mathrm{d}y}.
$$
We remark that $u\in\mathcal{L}^1_s$ is such that $u\in\mathcal{L}^1_{s+\delta}$ for any $\delta>0$, which will be an important tool in this section of the paper, in order to compute high order fractional Laplacians. Another definition of the fractional Laplacian, which can be constructed by a double change of variables as in \cite{guide}, is
$$
(-\Delta)^s u (x) = \frac{C(n,s)}{2} \int_{\R^n}{\frac{2u(x)-u(x+y)-u(x-y)}{\abs{y}^{n+2s}}\mathrm{d}y}
$$
which emphasize that given $u\in C^2(D)\cap \mathcal{L}^1_s$, we obtain that $x\mapsto (-\Delta)^s u(x)$ is a continuous and bounded function on $D$, for some bounded $D\subset \R^n$.\\\\
By \cite[Lemma 3.3]{MR2213639}, if we consider a regular unbounded cone $C$ symmetric with respect to a fixed axis, there exists two positive constant $c_1= c_1(n,s,C)$ and $c_2=c_2(n,s,C)$ such that
\begin{equation}\label{c(s)original}
c_1|x|^{\gamma_s-s}\mbox{dist}(x,\partial C)^s\leq u_s(x)\leq c_2|x|^{\gamma_s-s}\mbox{dist}(x,\partial C)^s
\end{equation}
for every $x \in C$. We remark that this result can be easily generalized to regular unbounded cones $C_\omega$ with $\omega\subset S^{n-1}$ which is a finite union of connected $C^{1,1}$ domain $\omega_i$, such that $\overline{\omega}_i \cup\overline{ \omega}_j = \emptyset$ for $i\neq j$, since the reasonings in \cite{MR2213639} rely on a Boundary Harnack principle and on sharp estimates for the Green function for bounded $C^{1,1}$ domain non necessary connected (for more details \cite{ChenSong}).\\\\
Through the paper we will call the coefficient of homogeneity $\gamma_s$ as "\emph{characteristic exponent}", since it is strictly related to an eigenvalue partition problem.\\
As we already mentioned, our solutions are smooth in the interior of the cone and locally $C^{0,s}$ near the boundary $\partial C\setminus \{0\}$ (see for example \cite{MR2213639}), but we need some quantitative estimates in order to better understand the dependence of the H$\ddot{\mbox{o}}$lder seminorm on the parameter $s\in (0,1)$. \\\\
Before showing the main result of H$\ddot{\mbox{o}}$lder regularity, we need the following estimates about the fractional Laplacian of smooth compactly supported functions: this result can be found in \cite[Lemma 3.5]{MR1671973} and \cite[Lemma 5.1]{concentration}, but here we compute the formula with a deep attention on the dependence of the constant with respect to $s\in (0,1)$.
\begin{Proposition}\label{labogdan}
Let $s\in(0,1)$ and $\varphi\in C^2_c(\mathbb{R}^n)$. Then
\begin{equation}
|(-\Delta)^s\varphi(x)|\leq\frac{c}{(1+|x|)^{n+2s}},\qquad\forall x\in\mathbb{R}^n,
\end{equation}
where the constant $c>0$ depends only on $n$ and the choice of $\varphi$.
\end{Proposition}
\proof
Let $K\subset\mathbb{R}^n$ be the compact support of $\varphi$ and $k=\max_{x\in K}|\varphi(x)|$. There exists $R>1$ such that $K\subset B_{R/2}(0)$.\\
Let $|x|>R$.
\begin{eqnarray*}
|(-\Delta)^s\varphi(x)|&=&\left|C(n,s)\int_{\mathbb{R}^n}\frac{\varphi(x)-\varphi(y)}{|x-y|^{n+2s}}\mathrm{d}y\right|=\left|C(n,s)\int_{K}\frac{\varphi(y)}{|x-y|^{n+2s}}\mathrm{d}y\right|\nonumber\\
&\leq&\frac{C(n,s)k}{|x|^{n+2s}}\int_{K}\frac{1}{(1-\left|\frac{y}{x}\right|)^{n+2s}}\mathrm{d}y\leq\frac{C(n,s)k2^{n+2s}|K|}{|x|^{n+2s}}\nonumber\\
&\leq&\frac{C(n,s)k2^{2(n+2s)}|K|}{(1+|x|)^{n+2s}}\leq\frac{c}{(1+|x|)^{n+2s}},
\end{eqnarray*}
where $c>0$ depends only on $n$ and the choice of $\varphi$.\\\\
Let now $|x|\leq R$. We use the fact that any derivative of $\varphi$ of first and second order is uniformly continuous in the compact set $K$ and the fact that in $B_R(0)$ the function $(1+|x|)^{n+2s}$ has maximum given by $(1+R)^{n+2s}$. Hence there exist $0<\delta<1$ and a constant $M>0$, both depending only on $n$ and the choice of $\varphi$ such that
\begin{equation*}
|\varphi(x+z)+\varphi(x-z)-2\varphi(x)|\leq M|z|^2\qquad\forall z\in B_\delta(0).
\end{equation*}
Hence
\begin{eqnarray*}
|(-\Delta)^s\varphi(x)|&=&\left|C(n,s)\int_{\mathbb{R}^n\setminus B_\delta(x)}\frac{\varphi(x)-\varphi(y)}{|x-y|^{n+2s}}\mathrm{d}y+C(n,s)\int_{B_\delta(x)}\frac{\varphi(x)-\varphi(y)}{|x-y|^{n+2s}}\mathrm{d}y\right|\nonumber\\
&\leq& 2kC(n,s)\int_{\mathbb{R}^n\setminus B_\delta(x)}\frac{1}{|x-y|^{n+2s}}\mathrm{d}y+\frac{C(n,s)}{2}\int_{B_\delta(0)}\frac{|\varphi(x+z)+\varphi(x-z)-2\varphi(x)|}{|z|^{n+2s}}\mathrm{d}z\nonumber\\
&\leq& 2kC(n,s)\omega_{n-1}\int_\delta^{+\infty}r^{-1-2s}\mathrm{d}r+\frac{C(n,s)\omega_{n-1}M}{2}\int_0^\delta r^{1-2s}\mathrm{d}r\nonumber\\
&=& \frac{kC(n,s)\omega_{n-1}}{s\delta^{2s}}+\frac{C(n,s)\omega_{n-1}M\delta^{2-2s}}{4(1-s)}\nonumber\\
&\leq& \frac{c}{\delta^2}+c= c\frac{(1+|x|)^{n+2s}}{(1+|x|)^{n+2s}}\leq \frac{c(1+R)^{n+2}}{(1+|x|)^{n+2s}}=\frac{c}{(1+|x|)^{n+2s}},
\end{eqnarray*}
where $c>0$ depends only on $n$ and the choice of $\varphi$. This concludes the proof.
\endproof
By the previous calculations we have also the following result.
\begin{remark}\label{labogdancns}
Let $s\in(0,1)$ and $\varphi\in C^2_c(\mathbb{R}^n)$. Then there exists a constant $c=c(n,\varphi)>0$ and a radius $R=R(\varphi)>0$ such that
\begin{equation}
|(-\Delta)^s\varphi(x)|\leq c\frac{C(n,s)}{(1+|x|)^{n+2s}},\qquad\forall x\in\mathbb{R}^n\setminus B_R(0).
\end{equation}
\end{remark}

The following result provides interior estimates for the H\"older norm of our solutions.
\begin{Proposition}\label{propcomp}
Let $C$ be a cone and $K\subset C$ be a compact set and $s_0\in(0,1)$. Then there exist a constant $c>0$ and $\overline\alpha\in(0,1)$, both dependent only on $s_0,K,n,C$, such that
\begin{equation*}
||u_s||_{C^{0,\alpha}(K)} \leq c\left(1+\frac{C(n,s)}{2s-\gamma_s(C)}\right),
\end{equation*}
for any $\alpha\in(0,\overline\alpha]$ and any $s\in[s_0,1)$.
\end{Proposition}
By a standard covering argument, there exists a finite number of balls such that $K\subset \bigcup_{j=1}^{k}B_{r}(x_j)$, for a given radius $r>0$ such that $\bigcup_{j=1}^{k}\overline{B_{2r}(x_j)}\subset C$. Thus, it is enough to prove
\begin{Proposition}\label{localreg}
Let $\overline{B_{2r}(\overline x)}\subset C$ be a closed ball and $s_0\in(0,1)$. Then there exist a constant $c>0$ and $\overline\alpha\in(0,1)$, both dependent only on $s_0,r,\overline x,n,C$, such that
\begin{equation*}
||u_s||_{C^{0,\alpha}(\overline{B_r(\overline x)})} \leq c\left(1+\frac{C(n,s)}{2s-\gamma_s(C)}\right),
\end{equation*}
for any $\alpha\in(0,\overline\alpha]$ and any $s\in[s_0,1)$.
\end{Proposition}
In order to achieve the desired result, we need to estimate locally the value of the fractional Laplacian of $u_s$ in a ball compactly contained in the cone $C$.
\begin{Lemma}
Let $\eta\in C^\infty_c(B_{2r}(\overline x))$ be a cut-off function such that $0\leq\eta\leq 1$ with $\eta\equiv 1$ in $B_r(\overline x)$. Under the same assumptions of Proposition \ref{localreg},
\begin{equation*}
||(-\Delta)^s(u_s\eta)||_{L^\infty(B_{2r}(\overline x))}\leq C_0\left(1+\frac{C(n,s)}{2s-\gamma_s(C)}\right)
\end{equation*}
for any $s\in[s_0,1)$, where $C_0>0$ depends on $s_0,n,\overline x,r,C,$ and the choice of the function $\eta$.
\end{Lemma}
\proof
Let $R>1$ such that $\overline{B_{2r}(\overline x)}\subset B_{R/2}(0)$. Hence, let fix a point $x\in B_{2r}(\overline x)$. We can express the fractional Laplacian of $u_s\eta$ in the following way
\begin{eqnarray*}
(-\Delta)^s(u_s\eta)(x)&=&\eta(x)(-\Delta)^su_s(x)+C(n,s)\int_{\mathbb{R}^n}u_s(y)\frac{\eta(x)-\eta(y)}{|x-y|^{n+2s}}\mathrm{d}y\nonumber\\
&=&C(n,s)\int_{B_R(0)}u_s(y)\frac{\eta(x)-\eta(y)}{|x-y|^{n+2s}}\mathrm{d}y+C(n,s)\int_{\mathbb{R}^n\setminus B_R(0)}u_s(y)\frac{\eta(x)-\eta(y)}{|x-y|^{n+2s}}\mathrm{d}y.
\end{eqnarray*}
We recall that $u_s(x)=|x|^{\gamma_s(C)}u_s(x/|x|)$ and that for any $s\in(0,1)$ the functions $u_s$ are normalized such that $||u_s||_{L^\infty(S^{n-1})}=1$. Moreover we remark that $\eta(x)-\eta(y)=\eta(x)\geq 0$ in $B_{2r}(\overline x)\times (\mathbb{R}^n\setminus B_R(0))$. Hence, using Proposition \ref{labogdan} and the fact that $\gamma_s(C)<2s$, we obtain\begin{eqnarray*}
|(-\Delta)^s(u_s\eta)(x)|&\leq&C(n,s)\left|\int_{B_R(0)}u_s(y)\frac{\eta(x)-\eta(y)}{|x-y|^{n+2s}}\mathrm{d}y\right|+C(n,s)\left|\int_{\mathbb{R}^n\setminus B_R(0)}u_s(y)\frac{\eta(x)-\eta(y)}{|x-y|^{n+2s}}\mathrm{d}y\right|\nonumber\\
&\leq& R^{\gamma_s(C)}|(-\Delta)^s\eta(x)|+C(n,s)2^{n+2s}\int_{\mathbb{R}^n\setminus B_R(0)}\frac{1}{|y|^{n+2s-\gamma_s(C)}}\mathrm{d}y\nonumber\\
&\leq& \frac{cR^2}{(1+|x|)^{n+2s}}+C(n,s)2^{n+2}\omega_{n-1}\int_R^{+\infty}r^{-1-2s+\gamma_s(C)}\mathrm{d}r\nonumber\\
&\leq& \frac{cR^2}{(1+|x|)^{n+2s}} +\frac{cC(n,s)}{R^{2s-\gamma_s(C)}(2s-\gamma_s(C))}\nonumber\\
&\leq& C_0\left(1+\frac{C(n,s)}{2s-\gamma_s(C)}\right).
\end{eqnarray*}
\endproof
\begin{proof}[Proof of Proposition \ref{localreg}]
Let as before $\eta\in C^\infty_c(B_{2r}(\overline x))$ be a cut-off function such that $0\leq\eta\leq 1$ with $\eta\equiv 1$ in $B_r(\overline x)$.
First, we remark that there exists a constant $c_0>0$ such that for any $s\in(0,1)$, it holds
\begin{equation}\label{bound}
||u_s\eta||_{L^\infty(\mathbb{R}^n)}\leq c_0,
\end{equation}
where $c_0$ depends only on $n,\overline x,r$. In fact, let $R>0$ be such that $\overline{B_{2r}(\overline x)}\subset B_R(0)$. Then, for any $x\in\mathbb{R}^n$, we have $0\leq u_s\eta(x)\leq R^{\gamma_s(C)}\leq R^2$. Using the bound \eqref{bound} and the previous Lemma, we can apply \cite[Theorem 12.1]{MR2494809} obtaining the existence of $\overline\alpha\in(0,1)$ and $C>0$, both depending only on $n,s_0$ and the choice of $B_r(\overline x)$ such that
\begin{align*}
||u_s\eta||_{C^{0,\alpha}(\overline{B_r(\overline x)})}&\leq C(||u_s\eta||_{L^\infty(\mathbb{R}^n)}+||(-\Delta)^s(u_s\eta)||_{L^\infty(B_{2r}(\overline x))})\\
&\leq C\left(c_0+C_0\left(1+\frac{C(n,s)}{2s-\gamma_s(C)}\right)\right),
\end{align*}
for any $s\in[s_0,1)$ and any $\alpha\in(0,\overline\alpha]$. Since $\eta\equiv 1$ in $B_r(\overline x)$ we obtain the result.
\end{proof}
Similarly, now we need to construct some estimate related to the $H^s$ seminorm of the solution $u_s$, Since the functions do not belong to $H^s(\R^n)$, we need to truncate the solution with some cut off function in order to avoid the problems related to the growth at infinity. In such a way, we can use
\begin{equation}\label{hsnorm}
[v]^2_{H^s(\R^n)} = \norm{(-\Delta)^{s/2}v}{L^2(\R^n)}^2 = \int_{\R^n}{v(-\Delta)^s v \mathrm{d}x}.
\end{equation}
which holds for every $v \in H^s(\R^n)$. So, let $\eta\in C^\infty_c(B_2)$ be a radial cut off function such that $\eta\equiv 1 $ in $B_1$ and $0\leq\eta\leq1$ in $B_2$, and consider $\eta_R(x) = \eta(\frac{x-x_0}{R})$ the rescaled cut off function defined in $B_{2R}(x_0)$, for some $R>0$ and $x_0\in\R^n$.
\begin{Proposition}\label{propHs}
Let $s_0\in(0,1)$ and $\eta_R \in C^\infty_c(B_{2R}(x_0))$ previously defined. Then
$$
[u_s\eta_R]^2_{H^s(\R^n)} \leq c\left(1 + \frac{C(n,s)}{2s - \gamma_s(C)} \right) 
$$
for any $s\in[s_0,1)$, where $c>0$ is a constant that depends on $x_0,R,C,s_0$ and $\eta$.
\end{Proposition}
\begin{proof}
Let $\eta\in C^\infty_c (B_2)$ be a radial cut off function such that $\eta\equiv 1$ in $B_1$ and $0\leq\eta\leq 1$ in $B_2$, and consider the collection of $(\eta_R)_R$ with $R>0$ defined by $\eta_R(x)=\eta(\frac{x-x_0}{R})$ with some $x_0\in\R^n$. By \eqref{hsnorm}, for every $R>0$ we obtain
$$
[u_s\eta_R]^2_{H^s(\R^n)} = \norm{(-\Delta)^{s/2}(u_s\eta_R)}{L^2(\R^n)}^2 = \int_{\R^n}{u_s\eta_R(-\Delta)^s (u_s\eta_R) \mathrm{d}x}.
$$
By definition of the fractional Laplacian we have
\begin{align*}
\int_{\R^n}{u_s\eta_R(-\Delta)^s (u_s\eta_R) \mathrm{d}x} &= C(n,s)\int_{\R^n\times\R^n}{u_s(x)\eta_{R}(x)\frac{u_s(x)\eta_R(x) - u_s(y)\eta_R(y)}{\abs{x-y}^{n+2s}}\mathrm{d}y\mathrm{d}x}\\
&= \int_{\R^n}{\eta^2_R u_s (-\Delta)^s u_s\mathrm{d}x}+C(n,s)\int_{\R^n\times\R^n}{\frac{\eta_R(x) - \eta_R(y)}{\abs{x-y}^{n+2s}}u_s(x)u_s(y)\eta_R(x)\mathrm{d}y\mathrm{d}x}\\
&= 
\frac{C(n,s)}{2}\int_{\R^n\times\R^n}{\frac{\abs{\eta_R(x) - \eta_R(y)}^2}{\abs{x-y}^{n+2s}}u_s(x)u_s(y) \mathrm{d}y\mathrm{d}x}
\end{align*}
where the last equation is obtained by the symmetrization of the previous integral with respect to the variable $(x,y) \in \R^n \times \R^n$. Before splitting the domain of integration into different subset, it is easy to see that
$$
\begin{aligned}
\eta_R(x) - \eta_R(y) &\equiv  0  & \mbox{in }& B_{R}(x_0)\times B_R(x_0) \cup (\R^n \setminus B_{2R}(x_0)) \times (\R^n \setminus B_{2R}(x_0)) \\
\abs{\eta_R(x) - \eta_R(y)} &\equiv  1 & \mbox{in }& B_{R}(x_0)\times (\R^n \setminus B_{2R}(x_0)) \cup (\R^n \setminus B_{2R}(x_0))\times B_R(x_0).
\end{aligned}
$$
where all the previous balls are centered at the point $x_0$. Hence, given the sets $\Omega_1 = B_{3R}(x_0)\times B_{3R}(x_0)$ and $\Omega_2 = B_{2R}(x_0)\times (\R^n \setminus B_{3R}(x_0)) \cup (\R^n\setminus B_{3R}(x_0)) \times B_{2R}(x_0)$ we have
\begin{align*}
\int_{\R^n\times\R^n}{\frac{\abs{\eta_R(x) - \eta_R(y)}^2}{\abs{x-y}^{n+2s}}u_s(x)u_s(y) \mathrm{d}y\mathrm{d}x}  \leq & \int_{\Omega_1}{\frac{\abs{\eta_R(x) - \eta_R(y)}^2}{\abs{x-y}^{n+2s}}u_s(x)u_s(y) \mathrm{d}y\mathrm{d}x}+\\
&+\int_{\Omega_2}{\frac{\abs{\eta_R(x) - \eta_R(y)}^2}{\abs{x-y}^{n+2s}}u_s(x)u_s(y) \mathrm{d}y\mathrm{d}x}.
\end{align*}
In particular
\begin{align*}
\int_{\Omega_1}{\frac{\abs{\eta_R(x) - \eta_R(y)}^2}{\abs{x-y}^{n+2s}}u_s(x)u_s(y) \mathrm{d}y\mathrm{d}x} &\leq \sup_{B_{3R}(x_0)} u_s^2\int_{B_{3R}(x_0)\times B_{3R}(x_0)}{\frac{\norm{\nabla \eta_R}{L^\infty(\R^n)}^2}{\abs{x-y}^{n+2s-2}}\mathrm{d}y\mathrm{d}x}\\
&\leq \norm{\nabla \eta_R}{L^\infty}^2\sup_{B_{3R}(x_0)} u_s^2\int_{B_{3R}(0)}{\mathrm{d}x \int_{B_{6R}(x)}{\frac{1}{\abs{x-y}^{n+2s-2}}\mathrm{d}y}}\\
&\leq \frac{\norm{\nabla \eta}{L^\infty}^2}{R^2}\sup_{B_{3R}(x_0)} u_s^2 \abs{B_{3R}}\abs{S^{n-1}} \frac{(6R)^{2-2s}}{2(1-s)}\\
&\leq C\norm{\nabla \eta}{L^\infty}^2 \frac{R^{n-2s}}{2(1-s)} \max\{\abs{x_0}^{2\gamma_s},(3R)^{2\gamma_s}\}\norm{u_s}{L^\infty(S^{n-1})}
\end{align*}
where in the second inequality we use the changes of variables $x-x_0$ and $y-x_0$ and the fact that $B_{3R}(0)\times B_{3R}(0) \subset B_{3R}(0)\times B_{6R}(x)$ for every $x \in B_{3R}(0)$. Similarly we have
\begin{align*}
\int_{\Omega_2}{\frac{\abs{\eta_R(x) - \eta_R(y)}^2}{\abs{x-y}^{n+2s}}u_s(x)u_s(y) \mathrm{d}y\mathrm{d}x} &\leq 2 \int_{B_{2R}(x_0)}{u_s(x)\left( \int_{\R^n\setminus B_{3R}(x_0)}{\frac{u_s(y)}{\abs{x-y}^{n+2s}} \mathrm{d}y}\right)\mathrm{d}x}\\
&\leq 2 \int_{B_{2R}(0)}{u_s(x+x_0)\left( \int_{\R^n\setminus B_{3R}(0)}{\frac{u_s(y+x_0)}{\abs{y}^{n+2s}\left(1-\frac{\abs{x}}{\abs{y}}\right)^{n+2s}} \mathrm{d}y}\right)\mathrm{d}x}\\
&\leq 2 \cdot 3^{n+2s}\int_{B_{2R}(0)}{u_s(x+x_0)\left( \int_{\R^n\setminus B_{3R}(0)}{\frac{C(\abs{y}+\abs{x_0})^{\gamma_s}}{\abs{y}^{n+2s}} \mathrm{d}y}\right)\mathrm{d}x}\\
&\leq C \sup_{B_{2R}(x_0)}{u_s} \abs{B_{2R}} \abs{S^{n-1}}2^{\gamma_s} G(x_0,R)
\end{align*}
with
$$
G(x_0,R)=\left\{
\begin{aligned}
\frac{\abs{x_0}^{\gamma_s}}{2s-\gamma_s}(3R)^{-2s}& \mbox{ if } \abs{x_0}\geq 3R\\
\frac{ (3R)^{\gamma_s-2s}}{2s-\gamma_s} & \mbox{ if } \abs{x_0}\leq 3R
\end{aligned}\right.
\leq \frac{(3R)^{-2s}}{2s-\gamma_s}\max\{\abs{x_0},3R\}^{\gamma_s}.
$$
Finally, we obtain the desired bound for the seminorm $[u_s\eta_R]^2_{H^s(\R^n)}$ summing the two terms and recalling that $\norm{u_s}{L^\infty(S^{n-1})}=1$.
\end{proof}

\section{Characteristic exponent $\gamma_s(C)$: properties and asymptotic behaviour}
In this section we start the analysis of the asymptotic behaviour of the homogeneity degree $\gamma_s(C)$ as $s$ converges to $1$. The main results are two: first we get a monotonicity result for the map $s\mapsto \gamma_s(C)$, for a fixed regular cone $C$, which ensures the existence of the limit and, using some comparison result, a bound on the possible value of the limit exponent. Secondly we study the asymptotic behaviour of  the quotient~$\frac{C(n,s)}{2s-\gamma_s(C)}$. \\\\
In order to prove the first result and compare different order of $s$-harmonic functions for different power of $(-\Delta)^s$, we need to introduce some results which give a natural extension of the classic semigroup property of the fractional Laplacian, for function defined on cones which grow at infinity.

\subsection{Distributional semigroup property}
It is well known that if we deal with smooth functions with compact support, or more generally with functions in the Schwartz space $ \mathcal{S}(\R^n)$, a semigroup property holds for the fractional Laplacian, i.e. $(-\Delta)^{s_1}\circ(-\Delta)^{s_2}=(-\Delta)^{s_1+s_2}$, where $s_1,s_2\in(0,1)$ with $s_1+s_2<1$. Since we have to deal with functions in $\mathcal{L}^1_s$ that grow at infinity, we have to construct a distributional counterpart of the semigroup property, in order to compute high order fractional Laplacians for solutions of the problem given in \eqref{Pstheta}.\\\\
First of all, we remark that a solution $u_s$ to \eqref{Pstheta} for a fixed cone $C$ belongs to $\mathcal{L}^1_s$ since $0\leq u_s(x)\leq|x|^{\gamma_s(C)}$ in $\mathbb{R}^n$ with $\gamma_s(C)\in(0,2s)$. Moreover, by the homogeneity one can rewrite the norm \eqref{norml1s} in the following way
\begin{align*}
\norm{u_s}{\mathcal{L}^1_s} = \int_{\R^n}{\frac{u_s(x)}{(1+\abs{x})^{n+2s}}\mathrm{d}x}=&\int_{S^{n-1}}{u_s\mathrm{d}\sigma}\int_{0}^\infty{\frac{\rho^{n-1+\gamma_s(C)}}{(1+\rho)^{n+2s}}\mathrm{d}\rho}\\
=&\frac{\Gamma(n+\gamma_s(C))\Gamma(2s-\gamma_s(C))}{\Gamma(n+2s)}\int_{S^{n-1}}{u_s\mathrm{d}\sigma}.
\end{align*}
In the recent paper \cite{newdefinition} the authors introduced a new notion of fractional Laplacian applying to a
wider class of functions which grow more than linearly at infinity. This is achieved by defining an equivalence class of functions modulo polynomials of a fixed order. However, it can be hardly exploited to the solutions of \eqref{Pstheta} as they annihilate on a set of nonempty interior.\\\\
As shown in \cite[Definition 3.6]{MR1671973}, if we consider a smooth function with compact support $\varphi\in C^\infty_c(\R^n)$(or $\varphi \in C^2_c (\R^n)$), we can define the distribution $k^{2s}$ by the formula
$$
(-\Delta)^s \varphi(0)= (k^{2s},\varphi).
$$
By this definition, it follows that $(-\Delta)^s \varphi(x) = k^{2s} \ast \varphi(x)$.
\begin{Definition}\cite[Definition 3.7]{MR1671973}
For $u\in \mathcal{L}^1_s$ we define the \emph{distributional fractional Laplacian} $(-\widetilde\Delta)^s u$ by the formula
$$
((-\widetilde\Delta)^s u, \varphi) = (u, (-\Delta)^s \varphi),\quad \forall\varphi \in C^\infty_c (\R^n).
$$
\end{Definition}
In particular, since given an open subset $D\subset \R^n$ and $u\in C^2(D)\cap \mathcal{L}^1_s$, the fractional Laplacian exists as a continuous function of $x\in D$ and $(-\widetilde\Delta)^s u = (-\Delta)^s u$ as a distribution in $D$ \cite[Lemma 3.8]{MR1671973}, through the paper we will always use $(-\Delta)^s$ both for the classic and the distributional fractional Laplacian. The following is a useful tool to compute the distributional fractional Laplacian.
\begin{Lemma}\cite[Lemma 3.3]{MR1671973}\label{epsilon}
Assume that
\begin{equation}\label{assumption}
\iint_{\abs{y-x}>\varepsilon}{\frac{\abs{f(x)g(y)}}{\abs{y-x}^{n+2s}}\mathrm{d}x\mathrm{d}y}<+\infty \quad and \quad\int_{\R^n}{\abs{f(x)g(x)}\mathrm{d}x}<+\infty,
\end{equation}
then $((-\Delta)^s_\varepsilon f, g)=(f, (-\Delta)^s_\varepsilon g)$. Moreover if $f\in \mathcal{L}^1_s$ and $g \in C_c(\R^n)$ the assumptions \eqref{assumption} are satisfied for every $\varepsilon>0$.
\end{Lemma}
Before proving the semigroup property, we prove the following lemma which ensures the existence of the $\delta$-Laplacian of the $s$-Laplacian, for $0<\delta<1$.
\begin{Lemma}\label{lemma.estimate}
Let $u_s$ be solution of \eqref{Pstheta} with $C$ a regular cone. Then we have $(-\Delta)^s u_s \in \mathcal{L}^1_\delta$ for any $\delta>0$, i.e.
$$
\int_{\R^n}{\frac{\abs{(-\Delta)^s u_s (x)}}{(1+\abs{x})^{n+2\delta}}\mathrm{d}x} < +\infty.
$$
\end{Lemma}
\begin{proof}
Since the function $u_s$ is $s$-harmonic in $C$, namely $(-\Delta)^s u_s (x)=0$ for all $x\in C$, we can restrict the domain of integration to $\R^n \setminus C$.\\
By homogeneity and the results in \cite{MR1671973}, we have that the function $(-\Delta)^s u_s$ is $(\gamma_s-2s)$-homogeneous and in particular $x\mapsto (-\Delta)^s u_s(x)$ is a continuous negative function, for every $x\in D \subset\subset \R^n\setminus C$. In order to compute the previous integral, we focus our attention on the restriction of the fractional Laplacian to the sphere $S^{n-1}$, in particular, we prove that there exists $\bar\varepsilon>0$ and $C>0$ such that \begin{equation}\label{estimate}
\abs{(-\Delta)^s u_s(x) }\leq \frac{C}{\mbox{dist}(x,\partial C)^s}\quad \forall x \in N_{\bar\varepsilon}(\partial C)\cap S^{n-1}, 
\end{equation}
where $N_\varepsilon(\partial C)=\{x \in \R^n\setminus C : \mbox{dist}(x,\partial C) \leq \varepsilon\}$ is the tubular neighborhood of $\partial C$ .\\
Hence, fixed $R>0$ small enough, consider initially $\varepsilon<R$ and $x\in S^{n-1}\cap N_\varepsilon(\partial C)$: since $u_s(y) \leq  \abs{y}^{\gamma_s}$ in $\mathbb{R}^n$ and by \eqref{c(s)original} there exists a constant $C>0$ such that for every $y\in C$ we have
$$ u_s(y) \leq C \abs{y}^{\gamma_s -s}\mbox{dist}(y,\partial C)^s, $$
it follows, defining $\delta(x):= \mbox{dist}(x,\partial C)>0$, that
\begin{align*}
\abs{(-\Delta)^s u_s (x)} &= C(n,s)\int_{C\cap B_R(x)}{\frac{u_s(y)}{\abs{x-y}^{n+2s}}\mathrm{d}y} + C(n,s)\int_{C\setminus B_R (x)}{\frac{u_s(y)}{\abs{x-y}^{n+2s}}\mathrm{d}y}\\
 &\leq C(n,s)\int_{C\cap B_R(x)}{\frac{C \abs{y}^{\gamma_s-s} \mbox{dist}(y,\partial C)^s}{\abs{x-y}^{n+2s}}\mathrm{d}y} + C(n,s)\int_{C\setminus B_R (x)}{\frac{\abs{y}^{\gamma_s}}{\abs{x-y}^{n+2s}}\mathrm{d}y}.
\end{align*}
Since $C\cap B_R(x) \subset B_{R}(x) \setminus B_{\delta(x)}(x)$, we have
\begin{align*}
\abs{(-\Delta)^s u_s(x)}
&\leq
C\int_{R\geq\abs{x-y}\geq\delta(x)}{\frac{\abs{y}^{\gamma_{s}-s}}{\abs{x-y}^{n+s}}\mathrm{d}y}+
\int_{\abs{x-y}\geq R}{\frac{(\abs{x-y}+1)^{\gamma_{s}}}{\abs{x-y}^{n+2s}}\mathrm{d}y}\\
&\leq C\int_{R\geq\abs{x-y}\geq\delta(x)}{\frac{1}{\abs{x-y}^{n+s}}\mathrm{d}y}+
\omega_{n-1}\int^\infty_{R}{\frac{(t+1)^{\gamma_{s}}}{t^{1+2s}}\mathrm{d}t}\\
&\leq C\int_{\delta(x)}^R{\frac{1}{r^{1+s}}\mathrm{d}r} + M\\
& \leq C\frac{1}{\mbox{dist}(x,\partial C)^s} + M.
\end{align*}
Moreover, again since $s\in (0,1)$, up to consider a smaller neighborhood $N_\varepsilon (\partial C)$, we obtain that there exists a constant $\bar\varepsilon>0$ small enough and $C>0$ such that
$$
\abs{(-\Delta)^s u_s(x)}\leq \frac{C}{\mbox{dist}(x,\partial C)^s} \quad \mbox{for every } x \in N_{\bar\varepsilon}(\partial C) \cap S^{n-1}.
$$
Now, fixed $\delta>0$ and considered $\bar\varepsilon>0$ of \eqref{estimate}, we have
\begin{align*}
\int_{\R^n\setminus C}{\frac{\abs{(-\Delta)^s u_s (x)}}{(1+\abs{x})^{n+2\delta}}\mathrm{d}x} &= \int_{\R^n\setminus C}{\frac{\abs{x}^{\gamma_s-2s}\abs{(-\Delta)^s u_s\Big(\frac{x}{\abs{x}} \Big)}}{(1+\abs{x})^{n+2\delta}}\mathrm{d}x}\\
& = \int_{0}^{\infty}{\int_{S^{n-1}\cap (\mathbb{R}^n\setminus C)}{\frac{r^{\gamma_s -2s}\abs{(-\Delta)^s u_s(z)}}{(1+r)^{n+2\delta}}r^{n-1}\mathrm{d}\sigma(z)}\mathrm{d}r}\\
&= \int_0^\infty{\frac{r^{n-1+\gamma_s-2s}}{(1+r)^{n+2\delta}}\mathrm{d}r}\int_{S^{n-1}\cap (\mathbb{R}^n\setminus C)}{\abs{(-\Delta)^s u_s(z)}\mathrm{d}\sigma}.
\end{align*}
Since $\gamma_s \in (0,2s)$ and $s\in (0,1)$, it follows
\begin{align*}
\int_{\R^n\setminus C}{\frac{\abs{(-\Delta)^s u_s (x)}}{(1+\abs{x})^{n+2\delta}}\mathrm{d}x} &\leq  C\int_{S^{n-1}\cap N_{\bar\varepsilon}(\partial C)}{\abs{(-\Delta)^s u_s(z)}\mathrm{d}\sigma}+C\int_{((\R^n\setminus C)\setminus N_{\bar\varepsilon}(\partial C))\cap S^{n-1}}{\abs{(-\Delta)^s u_s(z)}\mathrm{d}\sigma}\\
&\leq C\int_{S^{n-1}\cap N_{\bar\varepsilon}(\partial C)}{\frac{1}{\mbox{dist}(z,\partial C)^s}\mathrm{d}\sigma}+M\\
&< +\infty
\end{align*}
where in the second inequality we used that $z\mapsto (-\Delta)^s u_s(z)$ is continuous in every $A\subset \subset S^{n-1}\cap (\R^n\setminus C)$ and in the last one that $\mbox{dist}(x,\partial C)^{-s}\in L^1(S^{n-1}\cap N_{\bar\varepsilon}(\partial C),\mathrm{d}\sigma)$.
\end{proof}
\begin{Proposition}[Distributional semigroup property]\label{semigroup}
Let $u_s$ be a solution of \eqref{Pstheta} with $C$ a regular cone and consider $\delta \in (0,1-s)$. Then
$$
(-\Delta)^{s+\delta}u_s = (-\Delta)^\delta [(-\Delta)^s u_s] \quad \mathrm{in}\ \mathcal{D}'(C)
$$
or equivalently
$$
( (-\Delta)^{s+\delta}u_s, \varphi ) = ((-\Delta)^\delta [(-\Delta)^s u_s], \varphi), \quad \forall \varphi \in C^\infty_c(C).
$$
\end{Proposition}
\begin{proof}
Since $\abs{u_s(x)}\leq \abs{x}^{\gamma_s}$, with $\gamma_s \in (0,2s)$, it is easy to see that $u_s \in \mathcal{L}^1_s \cap C^2(C)$. Moreover, as we have already remarked, if $u_s\in\mathcal{L}^1_s$ then $u_s\in\mathcal{L}^1_{s+\delta}$ for every $\delta >0$. In particular, $(-\Delta)^{s+\delta}u_s$ does exist and it is a continuous function of $x\in C$, for every $\delta \in (0,1-s)$.
By definition of the distributional fractional Laplacian, we obtain
$$
((-\Delta)^{s+\delta} u_s, \varphi) =(u_s,(-\Delta)^{s+\delta} \varphi),
$$
and since for $\varphi \in C^\infty_c(C)\subset\mathcal{S}(\mathbb{R}^n)$ in the Schwarz space, the classic semigroup property holds, we obtain that
$$
((-\Delta)^{s+\delta} u_s, \varphi) =(u_s,(-\Delta)^{s}[(-\Delta)^\delta \varphi]).
$$
On the other hand, since by Lemma \ref{lemma.estimate} we have 
$(-\Delta)^s u_s \in \mathcal{L}^1_\delta$, it follows
\begin{equation}\label{disteps}
((-\Delta)^\delta_\varepsilon[(-\Delta)^s u_s],\varphi)=((-\Delta)^s u_s, (-\Delta)^\delta_\varepsilon \varphi)
\end{equation}
for every $\varepsilon >0$. Since $(-\Delta)^su_s\in\mathcal{L}^1_\delta$ and $\varphi\in C^\infty_c(\mathbb{R}^n)$, the $\delta$-Laplacian of $(-\Delta)^su_s$ does exists in a distributional sense and hence the left hand side in \eqref{disteps} does converge to $((-\Delta)^\delta[(-\Delta)^s u_s],\varphi)$ as $\varepsilon\to0$. Moreover the right hand side in \eqref{disteps} does converge to $((-\Delta)^s u_s, (-\Delta)^\delta \varphi)$ by the dominated convergence theorem, using Proposition \ref{labogdan} and Lemma \ref{lemma.estimate} which give
$$\int_{\mathbb{R}^n}(-\Delta)^su_s(x)(-\Delta)^\delta_\varepsilon\varphi(x)\mathrm{d}x\leq\int_{\mathbb{R}^n}\frac{|(-\Delta)^su_s(x)|}{(1+|x|)^{n+2\delta}}\mathrm{d}x<+\infty.$$
By the previous remarks, 
$$
((-\Delta)^\delta[(-\Delta)^s u_s],\varphi)=((-\Delta)^s u_s, (-\Delta)^\delta \varphi).
$$
In order to conclude the proof of the distributional semigroup property, we need to show that
\begin{equation}\label{finalpart}
(u_s,(-\Delta)^s[(-\Delta)^{\delta}\varphi]) = ((-\Delta)^s u_s,(-\Delta)^{\delta}\varphi),
\end{equation}
which is not a trivial equality, since $(-\Delta)^\delta\varphi \in C^\infty(\R^n)$ is no more compactly supported.\\\\
Let $\eta\in C^\infty_c(B_2(0))$ be a radial cutoff function such that $\eta\equiv 1$ in $B_1(0)$ and $0\leq\eta\leq 1$ in $B_2(0)$, and define $\eta_R(x)=\eta(x/R)$, for $R>0$. Obviously, since $u_s\eta_R \in C_c(\mathbb{R}^n)$ and $(-\Delta)^\delta\varphi \in \mathcal{L}^1_s$, by Lemma \ref{epsilon} we have
\begin{equation}\label{epsR}
(u_s\eta_R,(-\Delta)^s_{\varepsilon}[(-\Delta)^{\delta}\varphi]) = ((-\Delta)^s_{\varepsilon} (u_s\eta_R),(-\Delta)^{\delta}\varphi)
\end{equation}
for every $\varepsilon,R>0$. First, for $R>0$ fixed, we want to pass to the limit for $\varepsilon\to0$. For the left hand side in \eqref{epsR}, we get the convergence to $(u_s\eta_R,(-\Delta)^s[(-\Delta)^{\delta}\varphi])$ since we can apply the dominated convergence theorem. In fact
$$\int_{\mathbb{R}^n}u_s\eta_R(-\Delta)^s_\varepsilon[(-\Delta)^\delta\varphi]\leq c\int_K(-\Delta)^{s+\delta}\varphi<+\infty,$$
where $K$ denotes the support of $u_s\eta_R$. For the right hand side in \eqref{epsR} we observe that, for any $x\in\mathbb{R}^n$
\begin{equation*}
(-\Delta)^s_\varepsilon(u_s\eta_R)(x)=\eta_R(x)(-\Delta)^s_\varepsilon u_s(x)+u_s(x)(-\Delta)^s_\varepsilon\eta_R(x)-I_\varepsilon(u_s,\eta_R)(x),
\end{equation*}
where
$$I_\varepsilon(u_s,\eta_R)(x)=C(n,s)\int_{\mathbb{R}^n\setminus B_\varepsilon(x)}\frac{(u_s(x)-u_s(y))(\eta_R(x)-\eta_R(y))}{|x-y|^{n+2s}}\mathrm{d}y.$$
Obviously the first term $((-\Delta)^s_\varepsilon u_s,\eta_R(-\Delta)^\delta\varphi)\to((-\Delta)^s u_s,\eta_R(-\Delta)^\delta\varphi)$ by definition of the  distributional $s$-Laplacian, since $u_s\in\mathcal{L}^1_s$ and $\eta_R(-\Delta)^\delta\varphi\in C^\infty_c(\mathbb{R}^n)$. The second term $(u_s(-\Delta)^s_\varepsilon\eta_R,(-\Delta)^\delta\varphi)\to(u_s(-\Delta)^s\eta_R,(-\Delta)^\delta\varphi)$ by dominated convergence, since
$$\int_{\mathbb{R}^n}{u_s(-\Delta)^s_\varepsilon\eta_R(-\Delta)^\delta\varphi \mathrm{d}x} \leq c\int_{\mathbb{R}^n}{\frac{u_s(x)}{(1+|x|)^{n+2s}}\mathrm{d}x}.$$ Finally, the last term $(I_\varepsilon(u_s,\eta_R),(-\Delta)^\delta\varphi)\to(I(u_s,\eta_R),(-\Delta)^\delta\varphi)$ by dominated convergence, since
$$\int_{\mathbb{R}^n}{I_\varepsilon(u_s,\eta_R)(-\Delta)^\delta\varphi\mathrm{d}x}\leq C\int_{\mathbb{R}^n}{\abs{(-\Delta)^\delta\varphi}\mathrm{d}x}
,
$$
which is integrable by Proposition \ref{labogdan}.
Finally, passing to the limit for $\varepsilon\to0$, from \eqref{epsR} we get 
\begin{equation}\label{epsR2}
(u_s\eta_R,(-\Delta)^s[(-\Delta)^{\delta}\varphi]) = ((-\Delta)^s(u_s\eta_R),(-\Delta)^{\delta}\varphi),
\end{equation}
for every $R>0$.\\\\
Now we want to prove \eqref{finalpart}, concluding this proof, by passing to the limit in \eqref{epsR2} for $R\to+\infty$. Since we know, by dominated convergence, that the left hand side converges to $(u_s,(-\Delta)^s(-\Delta)^\delta \varphi)$ for $R\to \infty$, we focus our attention on the other one.
At this point, we need to prove that for any $\varphi\in C^{\infty}_c(C)$,
\begin{equation}\label{convR}
\int_{\mathbb{R}^n}(-\Delta)^s(u_s\eta_R)(-\Delta)^\delta\varphi\longrightarrow\int_{\mathbb{R}^n}(-\Delta)^su_s(-\Delta)^\delta\varphi,
\end{equation}
as $R\to+\infty$. First of all, we remark that $(-\Delta)^s(u_s\eta_R)\to(-\Delta)^su_s$ in $L^1_{\mathrm{loc}}(\mathbb{R}^n)$. In fact, let $K\subset\mathbb{R}^n$ be a compact set. There exists $\overline r>0$ such that $K\subset B_{\overline r}$. Then, considering any radius $R>\overline r$, $\eta_R(x)=1$ for any $x\in K$. Hence, for any $R>\overline r$, using the fact that $u_s(x)=|x|^{\gamma_s}u_s(x/|x|)$, we obtain
\begin{align*}
\int_K|(-\Delta)^s(u_s\eta_R)(x)-(-\Delta)^su_s(x)|\mathrm{d}x&=\int_K\mathrm{d}x\left|C(n,s)\mbox{P.V}\int_{\mathbb{R}^n}\frac{u_s(x)\eta_R(x)-u_s(y)\eta_R(y)+u_s(y)-u_s(x)}{|x-y|^{n+2s}}\mathrm{d}y\right|\nonumber\\
&=C(n,s)\int_K\mathrm{d}x\left(\mbox{P.V}\int_{C\setminus B_{R}}\frac{u_s(y)[1-\eta_R(y)]}{|x-y|^{n+2s}}\mathrm{d}y\right)\\
&\leq C(n,s)\int_K\mathrm{d}x\left(\mbox{P.V}\int_{C\setminus B_{R}}\frac{|y|^{\gamma_s}}{(|y|-\overline r)^{n+2s}}\mathrm{d}y\right)\\
&\leq C(n,s)\int_K\mathrm{d}x\left(\mbox{P.V}\int_{C\setminus B_{R}}\frac{|y|^{\gamma_s}}{|y|^{n+2s}(1-\frac{\overline r}{R})^{n+2s}}\mathrm{d}y\right)\\
&= C\left(\frac{R}{R-\overline r}\right)^{n+2s}\lim_{\rho\to+\infty}\int_{R}^\rho\frac{1}{r^{2s-\gamma_s+1}}\mathrm{d}r\nonumber\\
&= C\left(\frac{R}{R-\overline r}\right)^{n+2s}\frac{1}{R^{2s-\gamma_s}}\longrightarrow 0,
\end{align*}
as $R\to+\infty$. Hence we obtain also pointwise convergence almost everywhere. Moreover, we can give the following expression
\begin{equation}
(-\Delta)^s(u_s\eta_R)(x)=\eta_R(x)(-\Delta)^su_s(x)+C(n,s)\mbox{P.V.}\int_{\mathbb{R}^n}u_s(y)\frac{\eta_R(x)-\eta_R(y)}{|x-y|^{n+2s}}\mathrm{d}y.
\end{equation}
We remark that $\eta_R(x)(-\Delta)^su_s(x)\to(-\Delta)^su_s(x)$ and $\int_{\mathbb{R}^n}u_s(y)\frac{\eta_R(x)-\eta_R(y)}{|x-y|^{n+2s}}\mathrm{d}y\to0$ pointwisely. Moreover we can dominate the first term in the following way
\begin{equation*}
\eta_R(x)(-\Delta)^su_s(x)\leq(-\Delta)^su_s(x),
\end{equation*}
and
\begin{equation*}
\int_{\mathbb{R}^n}(-\Delta)^su_s(x)(-\Delta)^\delta\varphi(x)\mathrm{d}x<+\infty
\end{equation*}
since $(-\Delta)^s u_s \in \mathcal{L}^1_\delta$ and using Proposition \ref{labogdan} over $\varphi \in C^\infty_c(C)$. In order to prove \eqref{convR}, we want to apply the dominated convergence theorem, and hence we need the following condition for any $R>0$
\begin{equation*}
I:=\left|\int_{\mathbb{R}^n}(-\Delta)^\delta\varphi(x)\left(\mbox{P.V.}\int_{\mathbb{R}^n}u_s(y)\frac{\eta_R(x)-\eta_R(y)}{|x-y|^{n+2s}}\mathrm{d}y\right)\mathrm{d}x\right|\leq c.
\end{equation*}
Therefore, we will obtain a stronger condition; that is, the existence of a value $k>0$ such that for any $R>1$
$$I\leq\frac{c}{R^k}.$$
We split the region of integration $\mathbb{R}^n\times\mathbb{R}^n$ into five different parts; that is,
\begin{equation*}
\Omega_1:=(\mathbb{R}^n\setminus B_{2R})\times\mathbb{R}^n, \ \Omega_2:=B_{2R}\times B_{2R}, \ \Omega_3:=(B_{2R}\setminus B_R)\times(B_{3R}\setminus B_{2R}),
\end{equation*}
\begin{equation*}
\Omega_4:=(B_{2R}\setminus B_R)\times(\mathbb{R}^n\setminus B_{3R}), \ \Omega_5:=B_R\times(\mathbb{R}^n\setminus B_{2R}).
\end{equation*}
First of all, we remark that $(-\Delta)^s\eta_R(x)=R^{-2s}(-\Delta)^s\eta(x/R)$ and also that $||(-\Delta)^s\eta||_{L^\infty(\mathbb{R}^n)}<+\infty$. For the first term, using the fact that $\eta_R(x)-\eta_R(y)=0$ if $(x,y)\in(\mathbb{R}^n\setminus B_{2R})\times(\mathbb{R}^n\setminus B_{2R})$
\begin{eqnarray*}
I_1&:=&\int_{\mathbb{R}^n\setminus B_{2R}}|(-\Delta)^\delta\varphi(x)|\left|\int_{\mathbb{R}^n}u_s(y)\frac{\eta_R(x)-\eta_R(y)}{|x-y|^{n+2s}}\mathrm{d}y\right|\mathrm{d}x\nonumber\\
&\leq & \int_{\R^n \setminus B_{2R}}{\abs{(-\Delta)^\delta \varphi(x)}\abs{\int_{B_{2R}}{u_s(y)\frac{\eta_R(x)-\eta_R(y)}{\abs{x-y}^{n+2s}}}}\mathrm{d}x}\nonumber\\
&\leq&\int_{\mathbb{R}^n\setminus B_{2R}}|(-\Delta)^\delta\varphi(x)|\left(\sup_{B_{2R}}u_s\right)|(-\Delta)^s\eta_R(x)|\mathrm{d}x\nonumber\\
&\leq&\frac{c}{R^{2s-\gamma_s}}\int_{\mathbb{R}^n}\frac{1}{(1+|x|)^{n+2\delta}}\mathrm{d}x\leq \frac{c}{R^{2s-\gamma_s}}.
\end{eqnarray*}
For the second term, using the fact that $\eta_R(x)-\eta_R(y)\geq 0$ if $(x,y)\in B_{2R}\times(\mathbb{R}^n\setminus B_{2R})$, we obtain as before
\begin{eqnarray*}
I_2&:=&\int_{B_{2R}}|(-\Delta)^\delta\varphi(x)|\left|\int_{B_{2R}}u_s(y)\frac{\eta_R(x)-\eta_R(y)}{|x-y|^{n+2s}}\mathrm{d}y\right|\mathrm{d}x\nonumber\\
&\leq&\int_{B_{2R}}|(-\Delta)^\delta\varphi(x)|\left(\sup_{B_{2R}}u_s\right)|(-\Delta)^s\eta_R(x)|\mathrm{d}x\nonumber\\
&\leq&\frac{c}{R^{2s-\gamma_s}}\int_{\mathbb{R}^n}\frac{1}{(1+|x|)^{n+2\delta}}\mathrm{d}x\leq \frac{c}{R^{2s-\gamma_s}}.
\end{eqnarray*}
For the third part
\begin{equation*}
I_3:=\int_{B_{2R}\setminus B_R}|(-\Delta)^\delta\varphi(x)|\left|\int_{B_{3R}\setminus B_{2R}}u_s(y)\frac{\eta_R(x)-\eta_R(y)}{|x-y|^{n+2s}}\mathrm{d}y\right|\mathrm{d}x,
\end{equation*}
we consider the following change of variables $\xi=x/R\in B_2\setminus B_1$ and $\zeta=y/R\in B_3\setminus B_2$. Hence, using the $\gamma_s$-homogeneity of $u_s$ and the definition of our cut-off functions, we obtain
\begin{equation*}
I_3\leq\frac{R^{2n}}{R^{n+2s-\gamma_s}}\iint_{(B_2\setminus B_1)\times(B_3\setminus B_2)}|(-\Delta)^\delta\varphi(R\xi)|u_s(\zeta)\frac{\eta(\xi)-\eta(\zeta)}{|\xi-\zeta|^{n+2s}}\mathrm{d}\xi\mathrm{d}\zeta.
\end{equation*}
We use the fact that $u_s\in C^{0,s}(B_3\setminus B_1)$ (see \eqref{c(s)original} proved in \cite{MR2213639}) and the cut off function $\eta\in\mathrm{Lip}(B_3\setminus B_1)$; that is, there exists a constant $c>0$ such that
\begin{equation}\label{sholdlip}
|u_s(\xi)-u_s(\zeta)|\leq c|\xi-\zeta|^{s}\qquad\mathrm{and}\qquad |\eta(\xi)-\eta(\zeta)|\leq c|\xi-\zeta|,
\end{equation}
for every $\xi,\zeta \in B_3\setminus B_1$. Hence,
\begin{eqnarray*}
I_3&\leq&
\frac{R^{2n}}{R^{n+2s-\gamma_s}}\iint_{(B_2\setminus B_1)\times(B_3\setminus B_2)}|(-\Delta)^\delta\varphi(R\xi)|\frac{\abs{u_s(\zeta)-u_s(\xi)}\abs{\eta(\xi)-\eta(\zeta)}}{|\xi-\zeta|^{n+2s}}\mathrm{d}\xi\mathrm{d}\zeta\nonumber\\
&&+ \ \frac{R^{2n}}{R^{n+2s-\gamma_s}}\iint_{(B_2\setminus B_1)\times(B_3\setminus B_2)}|(-\Delta)^\delta\varphi(R\xi)|u_s(\xi)\frac{\abs{\eta(\xi)-\eta(\zeta)}}{|\xi-\zeta|^{n+2s}}\mathrm{d}\xi\mathrm{d}\zeta\nonumber\\
&=&J_1+J_2.
\end{eqnarray*}
By \eqref{sholdlip}, we obtain
\begin{eqnarray*}
J_1&\leq& c\frac{R^{2n}}{R^{n+2s-\gamma_s}}\iint_{(B_2\setminus B_1)\times(B_3\setminus B_2)}|(-\Delta)^\delta\varphi(R\xi)|\frac{|\xi-\zeta|^{s+1}}{|\xi-\zeta|^{n+2s}}\mathrm{d}\xi\mathrm{d}\zeta\nonumber\\
&\leq& c\frac{R^{2n}}{R^{n+2s-\gamma_s}}\iint_{(B_2\setminus B_1)\times(B_3\setminus B_2)}\frac{1}{(1+R|\xi|)^{n+2\delta}}\frac{1}{|\xi-\zeta|^{n+s-1}}\mathrm{d}\xi\mathrm{d}\zeta\nonumber\\
&\leq& \frac{c}{R^{2s+2\delta-\gamma_s}}\iint_{(B_2\setminus B_1)\times(B_3\setminus B_2)}\frac{1}{|\xi-\zeta|^{n+s-1}}\mathrm{d}\xi\mathrm{d}\zeta\leq \frac{c}{R^{2s+2\delta-\gamma_s}}.
\end{eqnarray*}
Moreover, using other two changes of variable $(\xi,\zeta)\mapsto(\xi,\xi+h)$ and $(\xi,\zeta)\mapsto(\xi,\xi-h)$, we obtain
\begin{eqnarray*}
J_2&\leq& \frac{R^{2n}}{R^{n+2s-\gamma_s}}\iint_{(B_2\setminus B_1)\times(B_3\setminus B_2)}|(-\Delta)^\delta\varphi(R\xi)|u_s(\xi)\frac{\eta(\xi)-\eta(\zeta)}{|\xi-\zeta|^{n+2s}}\mathrm{d}\xi\mathrm{d}\zeta\nonumber\\
&\leq&\frac{R^{2n}}{R^{n+2s-\gamma_s}}\iint_{(B_2\setminus B_1)\times(B_3\setminus B_2)}\frac{1}{(1+R\abs{\xi})^{n+2\delta}}u_s(\xi)\frac{\eta(\xi)-\eta(\zeta)}{|\xi-\zeta|^{n+2s}}\mathrm{d}\xi\mathrm{d}\zeta\nonumber\\
&\leq& \frac{c}{R^{2s+2\delta-\gamma_s}}\iint_{(B_2\setminus B_1)\times(B_3\setminus B_2)}\frac{\eta(\xi)-\eta(\zeta)}{|\xi-\zeta|^{n+2s}}\mathrm{d}\xi\mathrm{d}\zeta\nonumber\\
&\leq& \frac{c}{R^{2s+2\delta-\gamma_s}}\iint_{(B_2\setminus B_1)\times B_2}\frac{2\eta(\xi)-\eta(\xi+h)-\eta(\xi-h)}{|h|^{n+2s}}\mathrm{d}\xi\mathrm{d} h\nonumber\\
&\leq& \frac{c}{R^{2s+2\delta-\gamma_s}}\left(c+\iint_{(B_2\setminus B_1)\times B_\varepsilon}\frac{<\nabla^2\eta(\xi)h,h>}{|h|^{n+2s}}\mathrm{d}\xi\mathrm{d} h\right)\nonumber\\
&\leq& \frac{c}{R^{2s+2\delta-\gamma_s}}\left(c+\iint_{(B_2\setminus B_1)\times B_\varepsilon}\frac{1}{|h|^{n+2s-2}}\mathrm{d}\xi\mathrm{d} h\right) \leq \frac{c}{R^{2s+2\delta-\gamma_s}}.
\end{eqnarray*}
For the fourth part
\begin{equation*}
I_4:=\int_{B_{2R}\setminus B_R}|(-\Delta)^\delta\varphi(x)|\left|\int_{\mathbb{R}^n\setminus B_{3R}}u_s(y)\frac{\eta_R(x)-\eta_R(y)}{|x-y|^{n+2s}}\mathrm{d}y\right|\mathrm{d}x,
\end{equation*}
we consider, as before, the following change of variables $\xi=x/R\in B_2\setminus B_1$ and $\zeta=y/R\in \mathbb{R}^n\setminus B_3$. Hence,
\begin{eqnarray*}
I_4&\leq& c\frac{R^{2n}}{R^{n+2s-\gamma_s}}\iint_{(B_2\setminus B_1)\times(\mathbb{R}^n\setminus B_3)}|(-\Delta)^\delta\varphi(R\xi)|\frac{|\zeta|^{\gamma_s}}{|\zeta-\xi|^{n+2s}}\mathrm{d}\xi\mathrm{d}\zeta\nonumber\\
&\leq& c\frac{R^{2n}}{R^{n+2s-\gamma_s}}\iint_{(B_2\setminus B_1)\times(\mathbb{R}^n\setminus B_3)}\frac{1}{(1+R|\xi|)^{n+2\delta}}\frac{|\zeta|^{\gamma_s}}{|\zeta-\frac{2\zeta}{|\zeta|}|^{n+2s}}\mathrm{d}\xi\mathrm{d}\zeta\nonumber\\
&\leq& \frac{c}{R^{2s+2\delta-\gamma_s}}\iint_{(B_2\setminus B_1)\times(\mathbb{R}^n\setminus B_3)}\frac{|\zeta|^{\gamma_s}}{|\zeta|^{n+2s}(1-\frac{2}{|\zeta|})^{n+2s}}\mathrm{d}\xi\mathrm{d}\zeta\nonumber\\
&\leq& \frac{c}{R^{2s+2\delta-\gamma_s}}\iint_{(B_2\setminus B_1)\times(\mathbb{R}^n\setminus B_3)}\frac{1}{|\zeta|^{n+2s-\gamma_s}}\mathrm{d}\xi\mathrm{d}\zeta\leq \frac{c}{R^{2s+2\delta-\gamma_s}}.
\end{eqnarray*}
Eventually, we consider the last term
\begin{equation*}
I_5:=\int_{B_R}|(-\Delta)^\delta\varphi(x)|\left|\int_{\mathbb{R}^n\setminus B_{2R}}u_s(y)\frac{\eta_R(x)-\eta_R(y)}{|x-y|^{n+2s}}\mathrm{d}y\right|\mathrm{d}x.
\end{equation*}
Hence we obtain
\begin{eqnarray*}
I_5&\leq& c\int_{B_R}|(-\Delta)^\delta\varphi(x)|\left(\int_{\mathbb{R}^n\setminus B_{2R}}\frac{|y|^{\gamma_s}}{|y-x|^{n+2s}}\mathrm{d}y\right)\mathrm{d}x\nonumber\\
&\leq& c\int_{B_R}|(-\Delta)^\delta\varphi(x)|\left(\int_{\mathbb{R}^n\setminus B_{2R}}\frac{|y|^{\gamma_s}}{|y-\frac{Ry}{|y|}|^{n+2s}}\mathrm{d}y\right)\mathrm{d}x\nonumber\\
&\leq& c\int_{B_R}|(-\Delta)^\delta\varphi(x)|\left(\int_{\mathbb{R}^n\setminus B_{2R}}\frac{|y|^{\gamma_s}}{|y|^{n+2s}(1-\frac{R}{|y|})^{n+2s}}\mathrm{d}y\right)\mathrm{d}x\nonumber\\
&\leq& c\int_{B_R}|(-\Delta)^\delta\varphi(x)|\left(\int_{\mathbb{R}^n\setminus B_{2R}}\frac{1}{|y|^{n+2s-\gamma_s}}\mathrm{d}y\right)\mathrm{d}x\nonumber\\
&\leq& c\left(\int_{\mathbb{R}^n}\frac{1}{(1+|x|)^{n+2\delta}}\mathrm{d}x\right)\left(\int_{2R}^{+\infty}\frac{1}{r^{1+2s-\gamma_s}}\mathrm{d}r\right)\nonumber\\
&=& c\left(\int_{\mathbb{R}^n}\frac{1}{(1+|x|)^{n+2\delta}}\mathrm{d}x\right)\left(\lim_{\rho\to+\infty}\int_{2R}^{\rho}\frac{1}{r^{1+2s-\gamma_s}}\mathrm{d}r\right)\nonumber\\
&\leq&\frac{c}{R^{2s-\gamma_s}}.
\end{eqnarray*}
Since $I\leq\sum_{i=1}^5I_i$, we obtain the desired result.
\end{proof}
At this point, fixed $s\in (0,1)$, by the distributional semigroup property we can compute easily high order fractional Laplacians $(-\Delta)^{s+\delta}$ viewing it as the $\delta$-Laplacian of the $s$-Laplacian.
\begin{Corollary}\label{coroll1}
Let $C$ be a regular cone. For every $\delta \in (0,1-s)$, the solution $u_{s}$ of \eqref{Pstheta} is $(s+\delta)$-superharmonic in $C$ in the sense of distribution, i.e.
$$
((-\Delta)^{s+\delta}u_s, \varphi) \geq 0
$$
for every test function $\varphi\in C^\infty_c (C)$ nonnegative in $C$.\\
Moreover, $u_s$ is also superharmonic in $C$ in the sense of distribution, i.e.
$$
(-\Delta u_s, \varphi) \geq 0
$$
for every test function $\varphi\in C^\infty_c (C)$ nonnegative in $C$.
\end{Corollary}
\begin{proof}
As said before, the facts that $u_s \in \mathcal{L}^1_{s+\delta}$ and $u_s\in C^2(A)$ for every $A\subset \subset C$ ensure the existence of the $(-\Delta)^{s+\delta} u_s$ and the continuity of the map $x\mapsto (-\Delta)^{s+\delta} u_s(x)$ for every $x\in A\subset\subset C$. Hence at this point, the only part we need to prove is the positivity of the $(s+\delta)$-Laplacian in the sense of the distribution, which is a direct consequence of the previous result.
Indeed, since $u_s$ is a solution of the problem \eqref{Pstheta}, by Proposition \ref{semigroup} we know that for every $\varphi\in C^\infty_c (C)$ we have
\begin{align*}
((-\Delta)^{s+\delta}u_s,\varphi) &= ( (-\Delta)^\delta [(-\Delta)^s u_s],\varphi) \\
&= \int_{C}{\varphi(x)\mbox{ P.V.}\!\int_{\R^n}{\frac{(-\Delta)^s u_s(x) - (-\Delta)^s u_s(y)}{\abs{x-y}^{n+2\delta}}\mathrm{d}y}\mathrm{d}x}.
\end{align*}
where $(-\Delta)^\delta [(-\Delta)^s u_s]$ is well defined since that $(-\Delta)^s u_s \equiv 0 \in C^2(A)$ for every $A\subset\subset C$ and, by Lemma \ref{lemma.estimate}, $(-\Delta)^s u_s \in \mathcal{L}^1_\delta$ for every $\delta \in (0,1-s)$.\\ Consider now nonnegative test function $\varphi\geq0 $ in $C$, since $(-\Delta)^s u_s(x) = 0$ for every $x\in C$, we have for every $x\in \R^n\setminus \overline{C}$
$$
(-\Delta)^s u_s(x)=-\int_{C}{\frac{u_s(y)}{\abs{x-y}^{n+2s}}\mathrm{d}y}\leq 0.
$$
Similarly,
\begin{align*}
( (-\Delta)^\delta [(-\Delta)^s u_s],\varphi) &= \int_{C}{\varphi(x)\int_{\R^n}{\frac{- (-\Delta)^s u_s(y)}{\abs{x-y}^{n+2\delta}}\mathrm{d}y}\mathrm{d}x}\geq 0,
\end{align*}
since the support of $\varphi$ is compact in the cone $C$, and so there exists $\varepsilon>0$ such that $|x-y|>\varepsilon$ in the above integral. We have obtained that for any $\delta\in(0,1-s)$ and any nonnegative $\varphi\in C^\infty_c(C)$
\begin{equation*}
((-\Delta)^{s+\delta}u_s,\varphi)\geq 0,
\end{equation*}
then, passing to the limit for $\delta\to 1-s$, the function $u_s$ is superharmonic in the distributional sense
\begin{equation*}
0\leq\lim_{\delta\to 1-s}((-\Delta)^{s+\delta}u_s,\varphi)=\lim_{\delta\to 1-s}(u_s,(-\Delta)^{s+\delta}\varphi)=(u_s,-\Delta\varphi)=(-\Delta u_s,\varphi).
\end{equation*}
\end{proof}

\subsection{Monotonicity of $s\mapsto\gamma_s(C)$}
The following proposition is a consequence of Corollary \ref{coroll1} and it follows essentially the proof of Lemma 2 in \cite{Bogdan.narrow}.
\begin{Proposition}\label{prop1}
For any fixed regular cone $C$ with vertex in $0$, the map $s\mapsto\gamma_s(C)$ is monotone non decreasing in $(0,1)$.
\end{Proposition}
\proof
Fixed the cone $C$, let us denote with $\gamma_s$ and $\gamma_{s+\delta}$ respectively the homogeneities of $u_{s}$ and $u_{s+\delta}$. Let us suppose by contradiction that $\gamma_s>\gamma_{s+\delta}$ for a $\delta\in(0,1-s)$, and let us consider the function
\begin{equation*}
h(x)=u_{s+\delta}(x)-u_s(x)\quad\mathrm{in} \ \mathbb{R}^n,
\end{equation*}
where $u_s$ is the homogeneous solution of \eqref{Pstheta} and $u_{s+\delta}$ is the unique, up to multiplicative constants, nonnegative nontrivial homogeneous and continuous in $\mathbb{R}^n$ solution for
\begin{equation*}
\begin{cases}
(-\Delta)^{s+\delta} u=0, & \mathrm{in}\quad C, \\
u=0, & \mathrm{in}\quad \mathbb{R}^n\setminus C,
\end{cases}
\end{equation*}
of the form
\begin{equation*}
u_{s+\delta}(x)=|x|^{\gamma_{s+\delta}}u_{s+\delta}\left(\frac{x}{|x|}\right).
\end{equation*}
The function $h$ is continuous in $\mathbb{R}^n$ and $h(x)=0$ in $\mathbb{R}^n\setminus C$. We want to prove that $h(x)\leq 0$ in $\mathbb{R}^n\setminus(C\cap B_1)$. Since $h=0$ outside the cone, we can consider only what happens in $C\setminus B_1$. As we already quoted, we have
\begin{equation}\label{c(s)}
c_1(s)|x|^{\gamma_s-s}\mbox{dist}(x,\partial C)^s\leq u_s(x)\leq c_2(s)|x|^{\gamma_s-s}\mbox{dist}(x,\partial C)^s,
\end{equation}
for any $x\in \overline{C}\setminus\{0\}$,
and there exist two constants $c_1(s+\delta),c_2(s+\delta)>0$ such that
\begin{equation*}
c_1(s+\delta)|x|^{\gamma_{s+\delta}-(s+\delta)}\mbox{dist}(x,\partial C)^{s+\delta}\leq u_{s+\delta}(x)\leq c_2(s+\delta)|x|^{\gamma_{s+\delta}-(s+\delta)}\mbox{dist}(x,\partial C)^{s+\delta}.
\end{equation*}
We can choose $u_s$ and $u_{s+\delta}$ so that $c:=c_1(s)=c_2(s+\delta)$ since they are defined up to a multiplicative constant. Then, for any $x\in C\setminus B_1$, since $|x|^{\gamma_{s+\delta}}\leq|x|^{\gamma_s}$, we have
\begin{equation}\label{h<0}
h(x)\leq c|x|^{\gamma_s}\mbox{dist}(x,\partial C)^s\left[\frac{\mbox{dist}(x,\partial C)^\delta}{|x|^\delta}-1\right]\leq 0.
\end{equation}
In fact, if we take $x$ such that $\mbox{dist}(x,\partial C)\leq 1$, then \eqref{h<0} follows by
\begin{equation*}
\frac{\mbox{dist}(x,\partial C)^\delta}{|x|^\delta}-1\leq \mbox{dist}(x,\partial C)^\delta-1\leq 0.
\end{equation*}
Instead, if we consider $x$ so that $\mbox{dist}(x,\partial C)> 1$, then $\mbox{dist}(x,\partial C)^\delta<|x|^\delta$ and hence \eqref{h<0} follows.\\\\
Now we want to show that there exists a point $x_0\in C\cap B_1$ such that $h(x_0)>0$. Let us take a point $\overline x\in S^{n-1}\cap C$ and let $\alpha:=u_{s+\delta}(\overline x)>0$ and $\beta:=u_s(\overline x)>0$. Hence, there exists a small $r>0$ so that $\alpha r^{\gamma_{s+\delta}}>\beta r^{\gamma_s}$, and so, taking $x_0$ with $|x_0|=r$ and so that $\frac{x_0}{|x_0|}=\overline x$, we obtain $h(x_0)>0$.\\\\
If we consider the restriction of $h$ to $\overline{C\cap B_1}$, which is continuous on a compact set, for the considerations done before and for the Weierstrass Theorem, there exists a maximum point $x_1\in C\cap B_1$ for the function $h$ which is global in $\mathbb{R}^n$ and is strict at least in a set of positive measure. Hence,
\begin{equation*}
(-\Delta)^{s+\delta}h(x_1)=C(n,s)\mbox{ P.V.}\int_{\mathbb{R}^n}\frac{h(x_1)-h(y)}{|x_1-y|^{n+2(s+\delta)}}\,\mathrm{d}y>0,
\end{equation*}
and since $(-\Delta)^{s+\delta}h$ is a continuous function in the open cone, there exists an open set $U(x_1)$ with $\overline{U(x_1)}\subset C$ such that
$$(-\Delta)^{s+\delta}h(x)>0\quad\forall x\in U(x_1).$$
But thanks to Corollary \ref{coroll1} we obtain a contradiction since for any nonnegative $\varphi\in C^\infty_c(U(x_1))$
\begin{equation*}\label{hsub}
((-\Delta)^{s+\delta}h,\varphi)=((-\Delta)^{s+\delta}u_{s+\delta},\varphi)-((-\Delta)^{s+\delta}u_s,\varphi)=-((-\Delta)^{s+\delta}u_s,\varphi)\leq 0.
\end{equation*}
\endproof
With the same argument of the previous proof we can show also the following useful upper bound.
\begin{Proposition}\label{prop2}
For any fixed regular cone $C$ with vertex in $0$ and any $s\in(0,1)$, $\gamma_s(C)\leq\gamma(C)$.
\end{Proposition}
\proof
Seeking a contradiction, we suppose that there exists $s\in(0,1)$ such that $\gamma_s>\gamma$. Hence we define the function
\begin{equation*}
h(x)=u(x)-u_s(x)\quad\mathrm{in} \ \mathbb{R}^n,
\end{equation*}
where $u_s$ and $u$ are respectively solutions to \eqref{Pstheta} and
\begin{equation}\label{P1theta}
\begin{cases}
-\Delta u=0, & \mathrm{in}\quad C, \\
u=0, & \mathrm{in}\quad \mathbb{R}^n\setminus C.
\end{cases}
\end{equation}
We recall that these solutions are unique, up to multiplicative constants, nonnegative nontrivial homogeneous and continuous in $\mathbb{R}^n$ of the form
\begin{equation*}
u(x)=|x|^{\gamma}u\left(\frac{x}{|x|}\right),\qquad u_s(x)=|x|^{\gamma_s}u_s\left(\frac{x}{|x|}\right).
\end{equation*}
for some $\gamma_s \in (0,2s)$ and $\gamma \in (0,+\infty)$.
The function $h$ is continuous in $\mathbb{R}^n$ and $h(x)=0$ in $\mathbb{R}^n\setminus C$. We want to prove that $h(x)\leq 0$ in $\mathbb{R}^n\setminus(C\cap B_1)$. Since $h=0$ outside the cone, we can consider only what happens in $C\setminus B_1$. So, there exist two constants $c_1(s),c_2(s)>0$ such that, for any $x\in \overline{C}\setminus\{0\}$, it holds \eqref{c(s)}. Moreover there exist two constants $c_1,c_2>0$ such that,
\begin{equation*}
c_1|x|^{\gamma-1}\mbox{dist}(x,\partial C)\leq u(x)\leq c_2|x|^{\gamma-1}\mbox{dist}(x,\partial C).
\end{equation*}
We can choose $u_s$ and $u$ so that $c:=c_1(s)=c_2$ since they are defined up to a multiplicative constant. Then, for any $x\in C\setminus B_1$, since $|x|^{\gamma}\leq|x|^{\gamma_s}$, we have
\begin{equation*}
h(x)\leq c|x|^{\gamma_s}\mbox{dist}(x,\partial C)^s\left[\frac{\mbox{dist}(x,\partial C)^{1-s}}{|x|^{1-s}}-1\right]\leq 0,
\end{equation*}
with the same arguments of the previous proof.\\\\
Now we want to show that there exists a point $x_0\in C\cap B_1$ such that $h(x_0)>0$. Let us take a point $\overline x\in S^{n-1}\cap C$ and let $\alpha:=u(\overline x)>0$ and $\beta:=u_s(\overline x)>0$. Hence, there exists a small $r>0$ so that $\alpha r^{\gamma}>\beta r^{\gamma_s}$, and so, taking $x_0$ with $|x_0|=r$ and so that $\frac{x_0}{|x_0|}=\overline x$, we obtain $h(x_0)>0$.\\\\
If we consider the restriction of $h$ to $\overline{C\cap B_1}$, which is continuous on a compact set, for the considerations done before and for the Weierstrass Theorem, there exists at least a maximum point in $C\cap B_1$ for the function $h$ which is global in $\mathbb{R}^n$. Moreover, since $h$ cannot be  constant on $C\cap B_1$ and it is of class $C^2$ inside the cone, there exists a global maximum $y\in C\cap B_1$ such that, up to a rotation, $\partial^2_{x_ix_i}h(y)\leq 0$ for any $i=1,...,n$ and $\partial^2_{x_jx_j}h(y)<0$ for at least a coordinate direction. Hence
\begin{equation*}
\Delta h(y)=\sum_{i=1}^n \partial^2_{x_ix_i}h(y)<0.
\end{equation*}
By the continuity of $\Delta h$ in the open cone, there exists an open set $U(y)$ with $\overline{U(y)}\subset C$ such that
$$\Delta h(x)<0\quad\forall x\in U(y).$$
Since, by Corollary \ref{coroll1} for any nonnegative $\varphi\in C^\infty_c(U(y))$
\begin{equation*}
(-\Delta u_s,\varphi)\geq 0,
\end{equation*}
hence
\begin{equation*}\label{hsuper}
(\Delta h,\varphi)=(\Delta u,\varphi)-(\Delta u_s,\varphi)=(-\Delta u_s,\varphi)\geq 0,
\end{equation*}
and this is a contradiction.
\endproof

\subsection{Asymptotic behavior of $\frac{C(n,s)}{2s-\gamma_s(C)}$}
Let us define for any regular cone $C$ the limit
\begin{equation*}
\mu(C)=\lim_{s\to 1^-}\frac{C(n,s)}{2s-\gamma_s(C)}\in[0,+\infty].
\end{equation*}
Obviously, thanks to the monotonicity of $s\mapsto\gamma_s(C)$ in $(0,1)$, this limit does exist, but we want to show that $\mu(C)$ can not be infinite. At this point, this situation can happen since $2s-\gamma_s(C)$ can converge to zero and we do not have enough information about this convergence. The study of this limit depends on the cone $C$ itself and so we will consider separately the case of wide cones and narrow cones, which are respectively when $\gamma(C)<2$ and when $\gamma(C)\geq2$. In this section, we prove this result just for regular cones, while in Section \ref{4} we will extend the existence of a finite limit $\mu(C)$ to any unbounded cone, without the monotonicity result of Proposition \ref{prop1}.

\subsubsection{Wide cones: $\gamma(C)<2$}
We remark that, fixed a wide cone $C\subset\mathbb{R}^n$, then there exists $\varepsilon>0$ and $s_0\in(0,1)$, both depending on $C$, such that for any $s\in[s_0,1)$
\begin{equation*}
2s-\gamma_s(C)\geq\varepsilon>0.
\end{equation*}
In fact we know that $s\mapsto\gamma_s(C)$ is monotone non decreasing in $(0,1)$ and $0<\gamma_s(C)\leq\gamma(C)<2$. Hence, defining $\overline{\gamma}(C)=\lim_{s\to 1}\gamma_s(C)\in(0,2)$ we can choose
\begin{equation*}\label{s_0}
s_0:=\frac{\overline{\gamma}(C)-2}{4}+1\in(1/2,1)\quad\mathrm{and}\quad\varepsilon:=\frac{2-\overline{\gamma}(C)}{2}>0,
\end{equation*}
obtaining
\begin{equation*}\label{varepsgamma}
2s-\gamma_s(C)\geq 2s_0-\overline{\gamma}(C)=\varepsilon>0.
\end{equation*}
As a consequence we obtain $\mu(C)=0$ for any wide cone.

\subsubsection{Narrow cones: $\gamma(C)\geq 2$}
Before addressing the asymptotic analysis for any regular cone, we focus our attention on the spherical caps ones with "small" aperture. Hence, let us fix $\theta_0\in(0,\pi/4)$ and for any $\theta\in(0,\theta_0]$, let
\begin{equation*}
\lambda_1(\theta):=\lambda_1(\omega_\theta)=\min_{\substack{
            u\in H^1_0(S^{n-1}\cap C_\theta)\\ u \neq 0}}\frac{\int_{S^{n-1}}\abs{\nabla_{S^{n-1}}u}^2\mathrm{d}\sigma}{\int_{S^{n-1}}u^2\mathrm{d}\sigma}.
\end{equation*}
We have that $\lambda_1(\theta)>2n$, and hence the following problem is well defined
\begin{equation}\label{mu0}
\mu_0(\theta):=\min_{\substack{
            u\in H^1_0(S^{n-1}\cap C_\theta)\\ u \neq 0}}\frac{\int_{S^{n-1}}\abs{\nabla_{S^{n-1}}u}^2-2nu^2\mathrm{d}\sigma}{\left(\int_{S^{n-1}}|u|\mathrm{d}\sigma\right)^2}.
\end{equation}
This number $\mu_0(\theta)$ is strictly positive and achieved by a nonnegative $\varphi\in H^1_0(S^{n-1}\cap C_\theta)\setminus\{0\}$ which is strictly positive on $S^{n-1}\cap C_\theta$ and is obviously solution to
\begin{equation}\label{spheric}
\begin{cases}
-\Delta_{S^{n-1}} \varphi=2n\varphi+\mu_0(\theta)\displaystyle\int_{S^{n-1}}\!\!\varphi\mathrm{d}\sigma & \mathrm{in}\quad S^{n-1}\cap C_\theta, \\
\varphi=0 & \mathrm{in}\quad S^{n-1}\setminus C_\theta,
\end{cases}
\end{equation}
where $-\Delta_{S^{n-1}}$ is the Laplace-Beltrami operator on the unitary sphere $S^{n-1}$.\\
Let now $v$ be the $0$-homogeneous extension of $\varphi$ to the whole of $\mathbb{R}^n$ and $r(x):=|x|$. Such a function will be solution to
\begin{equation}\label{0.ext}
\begin{cases}
-\Delta v=\frac{2nv}{r^2}+\frac{\mu_0(\theta)}{r^2}\displaystyle\int_{S^{n-1}}\!\!v \mathrm{d}\sigma & \mathrm{in}\quad C_\theta, \\
v=0 & \mathrm{in}\quad \mathbb{R}^n\setminus C_\theta.
\end{cases}
\end{equation}
Since the spherical cap $C_\theta\cap S^{n-1}$ is an analytic submanifold of $S^{n-1}$ and the data $(\partial C_\theta \cap S^{n-1},0,\partial_\nu \varphi)$ are not characteristic, by the classic theorem of Cauchy-Kovalevskaya we can extend the solution $\varphi$ of \eqref{spheric} to a function $\tilde{\varphi}$, which is defined in a enlarged cone and it satisfies
$$
\begin{cases}
-\Delta_{S^{n-1}} \tilde{\varphi}=2n\tilde{\varphi}+\mu_0(\theta)\displaystyle\int_{S^{n-1}}{\varphi\mathrm{d}\sigma}& \mathrm{in}\quad S^{n-1}\cap C_{\theta+\varepsilon}, \\
\tilde{\varphi}=\varphi & \mathrm{in}\quad S^{n-1}\cap C_\theta,
\end{cases}
$$
for some $\varepsilon>0$. As in \eqref{0.ext}, we can define $\tilde{v}$ as the $0$-homogenous extension of $\tilde{\varphi}$. Finally, we introduce the following function
\begin{equation*}
v_s(x):=r(x)^{\gamma_s^*(\theta)}v(x),
\end{equation*}
where the choice of the homogeneity exponent $\gamma_s^*(\theta)\in(0,2s)$ will be suggested by the following important result.
\begin{Theorem}\label{teobarriera}
Let $\theta \in (0,\theta_0]$, then there exists $s_0=s_0(\theta)\in(0,1)$ such that
\begin{equation*}
(-\Delta)^sv_s(x)\leq 0\qquad\mathrm{in} \ C_\theta,
\end{equation*}
for any $s\in[s_0,1)$.
\end{Theorem}
\begin{proof}
By the $\gamma_s^*(\theta)$-homogeneity of $v_s$, it is sufficient to prove that $(-\Delta)^s v_s \leq 0$ on $C_\theta \cap S^{n-1}$, since $x\mapsto (-\Delta)^s v_s$ is $(\gamma_s^*(\theta)-2s)$-homogenous. In order to ease the notations, through the following computations we will simply use $\gamma$ instead of $\gamma_s^*(\theta)$ and $o(1)$ for the terms which converge to zero as $s$ goes to $1$. Hence, for $x\in S^{n-1}\cap C_\theta$, we have
$$
(-\Delta)^s v_s (x) = \abs{x}^{\gamma}(-\Delta)^s v(x) + v(x)(-\Delta)^s r^{\gamma}(x) - C(n,s)\int_{\R^n}{\frac{(r^\gamma (x)-r^\gamma(y))(v(x)-v(y))}{\abs{x-y}^{n+2s}}\mathrm{d}y}.
$$
First for $R>0$,
\begin{align*}
(-\Delta)^s r^\gamma(x)= &\, C(n,s)\int_{B_R(x)}{\frac{\abs{x}^{\gamma}-\abs{y}^\gamma}{\abs{x-y}^{n+2s}}\mathrm{d}y}+C(n,s)\int_{\R^n\setminus B_{R}(x)}{\frac{\abs{x}^\gamma-\abs{y}^\gamma}{\abs{x-y}^{n+2s}}\mathrm{d}y}\\
=& \, \frac{C(n,s)}{2}\int_{B_R(0)}{\frac{2\abs{x}^\gamma-\abs{x+z}^\gamma-\abs{x-z}^\gamma}{\abs{z}^{n+2s}}\mathrm{d}z}+C(n,s)\int_{\R^n\setminus B_{R}(x)}{\frac{1-\abs{y}^\gamma}{\abs{x-y}^{n+2s}}\mathrm{d}y}\\
=&\, -\frac{C(n,s)}{2}\int_0^R{\frac{\rho^2\rho^{n-1}}{\rho^{n+2s}}\mathrm{d}\rho}\int_{S^{n-1}}{\langle \nabla^2 \abs{x}^\gamma z,z\rangle \mathrm{d}\sigma}+o(1)+\\
&\,+C(n,s)\abs{S^{n-1}}\int_R^\infty{\frac{1}{\rho^{1+2s}}\mathrm{d}\rho}-C(n,s)\int_{\R^n\setminus B_{R}(x)}{\frac{\abs{y}^\gamma}{\abs{x-y}^{n+2s}}\mathrm{d}y}\\
=&-\frac{C(n,s)}{2}\frac{R^{2-2s}}{2-2s}\int_{S^{n-1}}{\langle \nabla^2 \abs{x}^\gamma z,z\rangle \mathrm{d}\sigma}+\\
&\,-C(n,s)\int_R^\infty{\frac{\rho^{n-1+\gamma}}{\rho^{n+2s}} \int_{S^{n-1}}{\abs{\frac{x}{\rho}-\vartheta}^\gamma \mathrm{d\sigma}(\vartheta)}\mathrm{d}\rho} + o(1).
\end{align*}
Since for every symmetric matrix $A$ we have
$$
\int_{S^{n-1}}{\langle A z, z\rangle \mathrm{d}\sigma} = \frac{\mbox{tr}A}{n}\omega_{n-1}
$$
where $\omega_{n-1}$ is the Lebesgue measure of the $(n-1)$-sphere $S^{n-1}$, we can simplify the first term since $\mbox{tr}\nabla^2\abs{x}^\gamma = \Delta(\abs{x}^\gamma)$ and checking that $\abs{\frac{x}{\rho}-\vartheta}^\gamma = 1 + \gamma\rho^{-1}\langle \vartheta,x\rangle + o(\rho^{-1})$ as $\rho \to \infty$ it follows
\begin{align*}
(-\Delta)^s r^\gamma(x)
=&-\frac{C(n,s)}{2}\frac{R^{2-2s}}{2-2s}\frac{\Delta ( \abs{x}^\gamma)\omega_{n-1}}{n}-C(n,s)\omega_{n-1}\int_R^\infty{\frac{\rho^{n-1+\gamma}}{\rho^{n+2s}} \mathrm{d}\rho} + o(1)\\
=& -\frac{C(n,s)\omega_{n-1}}{4n(1-s)}\gamma(n-2+\gamma)\abs{x}^{\gamma-2}R^{2-2s} - \frac{C(n,s)}{2s-\gamma}\omega_{n-1}R^{\gamma-2s} + o(1)\\
=&-\frac{C(n,s)\omega_{n-1}}{4n(1-s)}\gamma(n-2+\gamma)R^{2-2s} - \frac{C(n,s)}{2s-\gamma}\omega_{n-1}R^{\gamma-2s} + o(1)\\
=&-\frac{C(n,s)\omega_{n-1}}{4n(1-s)}\gamma(n-2+\gamma) - \frac{C(n,s)}{2s-\gamma}\omega_{n-1} + o(1),
\end{align*}
where in the last equality we choose $\gamma=\gamma_s^*(\theta)$ such that $\gamma_s^*(\theta)-2s \to 0$ as $s$ goes to 1.\\
Similarly, if $\tilde{v}$ is the $0$-homogenous extension of $v$ in an enlarged cone, which is such that $v\geq \tilde{v}$ and $v=\tilde{v}$ on $C_\theta\cap S^{n-1}$, it follows
\begin{align*}
(-\Delta)^s v(x) =& \frac{C(n,s)}{2} \int_{\abs{z}<1}{\frac{2v(x)-v(x+z)-v(x-z)}{\abs{z}^{n+2s}}\mathrm{d}z}
+ C(n,s)\int_{\abs{x-y}>1}{\frac{v(x)-v(y)}{\abs{x-y}^{n+2s}}\mathrm{d}y}\\
\leq & \frac{C(n,s)}{2} \int_{\abs{z}<1}{\frac{2\tilde{v}(x)-\tilde{v}(x+z)-\tilde{v}(x-z)}{\abs{z}^{n+2s}}\mathrm{d}z}
+ C(n,s) \int_1^\infty{\frac{\rho^{n-1}}{\rho^{n+2s}}\int_{S^{n-1}}{v(x)-v(y)\mathrm{d}\sigma}\mathrm{d}\rho}\\
=& -\frac{C(n,s)}{2} \int_0^1{\frac{\rho^{n-1}\rho^2}{\rho^{n+2s}}\int_{S^{n-1}}{\langle \nabla^2 \tilde{v}(x) z,z\rangle \mathrm{d}\sigma}\mathrm{d}\rho}
+ o(1)\\
=& \frac{C(n,s)\omega_{n-1}}{4n(1-s)}(-\Delta)\tilde{v}(x)+ o(1),
\end{align*}
where we can use that $\tilde{v}$ solves
$$
-\Delta \tilde{v} = 2n\tilde{v} + \mu_0\int_{S^{n-1}}{v\mathrm{d}\sigma}
$$
in the enlarged cap $S^{n-1}\cap C_{\theta+\varepsilon}$. Finally,
\begin{align*}
C(n,s)\int_{\R^n}{\frac{(\abs{x}^\gamma-\abs{y}^\gamma)(v(x)-v(y))}{\abs{x-y}^{n+2s}}\mathrm{d}y}=&
C(n,s)\left[\int_{\abs{y}<1}{\frac{(1-\abs{y}^\gamma)(v(x)-v(y))}{\abs{x-y}^{n+2s}}\mathrm{d}y}+\right.\\
&\left.+\int_{\abs{y}>1}{\frac{(1-\abs{y}^\gamma)(v(x)-v(y))}{\abs{x-y}^{n+2s}}\mathrm{d}y}\right]
\end{align*}
where the first term is $o(1)$ since
\begin{align*}
\int_0^1{(1-\rho^\gamma)\rho^{n-1}\int_{S^{n-1}}{\frac{v(x)-v(y)}{\abs{x-\rho y}^{n+2s}} \mathrm{d}\sigma}\mathrm{d}\rho} =& \int_0^1{(1-\rho^\gamma)\rho^{n-1}\int_{S^{n-1}}{(v(x)-v(y))(1+o(\rho))\mathrm{d}\sigma}\mathrm{d}\rho}\\
&+ \int_0^R{(1-\rho^\gamma)\rho^{n-1}\int_{S^{n-1}}{(v(x)-v(y))(n+2s)\rho \langle x, y \rangle\mathrm{d}\sigma}\mathrm{d}\rho}.
\end{align*}
Hence, we obtain
\begin{align*}
C(n,s)\int_{\R^n}{\frac{(\abs{x}^\gamma-\abs{y}^\gamma)(v(x)-v(y))}{\abs{x-y}^{n+2s}}\mathrm{d}y}=&
C(n,s)\int_{\abs{y}>1}{\frac{(1-\abs{y}^\gamma)(v(x)-v(y))}{\abs{x-y}^{n+2s}}\mathrm{d}y}+o(1)\\
=&o(1)-C(n,s)\int_{\abs{y}>1}{\frac{\abs{y}^\gamma(v(x)-v(y))}{\abs{x-y}^{n+2s}}\mathrm{d}y}+o(1)\\
=&o(1)-C(n,s)\int_1^\infty{\rho^\gamma \rho^{n-1}\int_{S^{n-1}}{\frac{v(x)-v(y)}{\abs{x-\rho y}^{n+2s}}\mathrm{d}\sigma}\mathrm{d}\rho}\\
=&o(1) - C(n,s)\int_1^\infty{\rho^{-1+\gamma-2s}\int_{S^{n-1}}{(v(x)-v(y))(1+o(\rho^{-1}))\mathrm{d}\sigma}\mathrm{d}\rho}+\\
&-C(n,s)\int_1^\infty{\rho^{-1+\gamma -2s} \int_{S^{n-1}}{(v(x)-v(y))(n+2s)\langle y, x\rangle \rho^{-1}\mathrm{d}\sigma}\mathrm{d}\rho}\\
=& o(1) - \frac{C(n,s)\omega_{n-1}}{2s-\gamma} v(x) + \frac{C(n,s)}{2s-\gamma}\int_{S^{n-1}}{v(y)\mathrm{d}\sigma}.
\end{align*}
Hence, recalling that $\gamma=\gamma^*_s(\theta)$, for $x\in S^{n-1}\cap C_\theta$ we have
\begin{align*}
(-\Delta)^s v_s(x)
&\leq\left(\mu_0(\theta)\frac{C(n,s)\omega_{n-1}}{4n(1-s)} - \frac{C(n,s)}{2s-\gamma_s^*(\theta)}\right)\int_{S^{n-1}}{v_s\mathrm{d}\sigma}
 + \frac{C(n,s)\omega_{n-1}}{4n(1-s)}(n+\gamma_s^*(\theta))(2-\gamma^*_s(\theta))v_s\\
&\leq   \left(\mu_0(\theta) - \frac{C(n,s)}{2s-\gamma^*_s(\theta)}\right) \int_{S^{n-1}}{v_s \mathrm{d}\sigma} + o(1)
\end{align*}
where $o(1)$ is uniform with respect to $\gamma^*_s(\theta)$ as $s\to 1$. In order to obtain a negative right hand side, it is sufficient to choose $\gamma^*_s(\theta) < 2s$ in such a way to make the denominator $2s-\gamma^*_s(\theta)$ small enough and the quotient $\frac{C(n,s)}{2s-\gamma^*_s(\theta)}$ still bounded.
\end{proof}
The previous result suggestes the following choice of the homogeneity exponent
$$\gamma_s^*(\theta):=2s-s\frac{C(n,s)}{\mu_0(\theta)}.$$
We can finally prove the main result of this section.
\begin{Corollary}\label{corollmu}
For any regular cone $C$, $\mu(C)<+\infty$.
\end{Corollary}
\proof
We will show that $\mu(\theta)<+\infty$ for any $\theta\in(0,\theta_0]$. Then, fixed an unbounded regular cone $C$, there exists a spherical cone $C_\theta$ such that $\theta\in(0,\theta_0]$ and $C_\theta\subset C$. Since by inclusion $\gamma_s(C)<\gamma_s(\theta)$, we obtain
\begin{equation*}
\mu(C)\leq\mu(\theta)<+\infty.
\end{equation*}
We want to show that fixed $\theta\in(0,\theta_0]$, $\gamma_s(\theta)\leq\gamma_s^*(\theta)$ for any $s\in[s_0(\theta),1)$, where the choice of $s_0(\theta)\in(0,1)$ is given in Theorem \ref{teobarriera}. The proof of this fact is based on considerations done in Proposition \ref{prop1}. By contradiction, $\gamma_s(\theta)>\gamma_s^*(\theta)$. Let
$$h(x)=v_s(x)-u_s(x).$$
The function $h$ is continuous in $\mathbb{R}^n$ and $h(x)=0$ in $\mathbb{R}^n\setminus C_\theta$. We want to prove that $h(x)\leq 0$ in $\mathbb{R}^n\setminus(C_\theta\cap B_1)$. Since $h=0$ outside the cone, we can consider only what happens in $C_\theta\setminus B_1$. By \eqref{c(s)}, there exist two constants $c_1(s),c_2(s)>0$ such that, for any $x\in \overline{C_\theta}\setminus\{0\}$,
\begin{equation*}
c_1(s)|x|^{\gamma_s-s}\mbox{dist}(x,\partial C_\theta)^s\leq u_s(x)\leq c_2(s)|x|^{\gamma_s-s}\mbox{dist}(x,\partial C_\theta)^s,
\end{equation*}
and there exist two constants $c_1,c_2>0$ such that
\begin{equation*}
c_1|x|^{\gamma_s^*-1}\mbox{dist}(x,\partial C_\theta)\leq v_{s}(x)\leq c_2|x|^{\gamma_s^*-1}\mbox{dist}(x,\partial C_\theta).
\end{equation*}
We can choose $v_s$ so that $c:=c_1(s)=c_2$ since it is defined up to a multiplicative constant. Then, for any $x\in C_\theta\setminus B_1$, since $|x|^{\gamma_s^*}\leq|x|^{\gamma_s}$, we have
\begin{equation*}
h(x)\leq c|x|^{\gamma_s}\mbox{dist}(x,\partial C_\theta)^s\left[\frac{\mbox{dist}(x,\partial C_\theta)^{1-s}}{|x|^{1-s}}-1\right]\leq 0.
\end{equation*}
Now we want to show that there exists a point $x_0\in C_\theta\cap B_1$ such that $h(x_0)>0$. Let us consider for example the point $\overline x\in S^{n-1}\cap C_\theta$ determined by the angle $\vartheta=\theta/2$, and let $\alpha:=v_s(\overline x)>0$ and $\beta:=u_s(\overline x)>0$. Hence, there exists a small $r>0$ so that $\alpha r^{\gamma_s^*}>\beta r^{\gamma_s}$, and so, taking $x_0$ with angle $\vartheta=\theta/2$ and $|x_0|=r$, we obtain $h(x_0)>0$.\\\\
If we consider the restriction of $h$ to $\overline{C_\theta\cap B_1}$, which is continuous on a compact set, for the considerations done before and for the Weierstrass Theorem, there exists a maximum point $x_1\in C_\theta\cap B_1$ for the function $h$ which is global in $\mathbb{R}^n$ and is strict at least in a set of positive measure. Hence,
\begin{equation*}
(-\Delta)^sh(x_1)=C(n,s)\mbox{ P.V.}\int_{\mathbb{R}^n}\frac{h(x_1)-h(y)}{|x_1-y|^{n+2s}}\,\mathrm{d}y>0,
\end{equation*}
and since $(-\Delta)^sh$ is a continuous function in the open cone, there exists an open set $U(x_1)$ with $\overline{U(x_1)}\subset C_\theta$ such that
$$(-\Delta)^sh(x)>0\quad\forall x\in U(x_1).$$
But thanks to Theorem \ref{teobarriera} we obtain a contradiction since for any nonnegative $\varphi\in C^\infty_c(U(x_1))$
\begin{equation*}
((-\Delta)^sh,\varphi)=((-\Delta)^sv_s,\varphi)-((-\Delta)^su_s,\varphi)=((-\Delta)^sv_s,\varphi)\leq 0,
\end{equation*}
where the last inequality holds for any $s\in[s_0(\theta),1)$. Hence, for any $\theta\in(0,\theta_0]$
\begin{equation}\label{upperbound.mu}
\mu(\theta)=\lim_{s\to 1^-}\frac{C(n,s)}{2s-\gamma_s(\theta)}\leq\lim_{s\to 1^-}\frac{C(n,s)}{2s-\gamma_s^*(\theta)}=\mu_0(\theta)<+\infty.
\end{equation}
\endproof

\section{The limit for $s\nearrow 1$}\label{4}
In this section we prove the main result, Theorem \ref{teolimit1}, emphasizing the difference between wide and narrow cones. Then we improve the asymptotic analysis proving uniqueness of the limit under assumptions on the geometry and the regularity of $C$.\\\\
Let $C\subset \R^n$ be an open cone and consider the minimization problem
\begin{equation}\label{lambda1}
\lambda_1(C)=\inf\left\{\ddfrac{\int_{S^{n-1}}|\nabla_{S^{{n-1}}}u|^2\mathrm{d}\sigma}{\int_{S^{n-1}}u^2\mathrm{d}\sigma} \ : \ u\in H^1(S^{n-1})\setminus\{0\} \ \mathrm{and} \ u=0 \ \mathrm{in} \ S^{n-1}\setminus C\right\},
\end{equation}
which is strictly related to the homogeneity of the solution of \eqref{P1theta} by $\lambda_1(C)=\gamma(C) ( \gamma(C) + n -2)$.\\
Moreover, if $\gamma(C) >  2$, equivalently if $\lambda_1(C)> 2n$, the problem
\begin{equation}\label{mu01}
\mu_0(C):=\inf\left\{\ddfrac{\int_{S^{n-1}}\abs{\nabla_{S^{n-1}}u}^2-2nu^2\mathrm{d}\sigma}{\left(\int_{S^{n-1}}|u|\mathrm{d}\sigma\right)^2} \ : \ u\in H^1(S^{n-1})\setminus\{0\} \ \mathrm{and} \ u=0 \ \mathrm{in} \ S^{n-1}\setminus C\right\}
\end{equation}
is well defined and the number $\mu_0(C)$ is strictly positive. \\\\By a standard argument due to the variational characterization of the previous quantities, we already know the existence of a nonnegative eigenfunction $\varphi \in H^1_0(S^{n-1}\cap C)\setminus\{0\}$ associated to the minimization problem \eqref{lambda1} and a nonnegative function $\psi \in H^1_0(S^{n-1}\cap C)\setminus\{0\}$ that achieves the minimum \eqref{mu01}, since the numerator in \eqref{mu01} is a coercive quadratic form equivalent to the one in \eqref{lambda1}.\\\\
Since the cone $C$ may be disconnected, it is well known that $\varphi$ is not necessarily unique. Instead, the function $\psi$ is unique up to a multiplicative constant, since it solves
\begin{equation}\label{spheric1}
\begin{cases}
-\Delta_{S^{n-1}} \psi=2n\psi+\mu_0(C)\displaystyle\int_{S^{n-1}}\!\!\psi\mathrm{d}\sigma & \mathrm{in}\quad S^{n-1}\cap C, \\
\psi=0 & \mathrm{in}\quad S^{n-1}\setminus C.
\end{cases}
\end{equation}
In fact, due to the integral term in the equation, the solution $\psi$ must be strictly positive in every connected component of $C$ and localizing the equation in a generic component we can easily get uniqueness by maximum principle.\\\\
A fundamental toll in order to reach as $s\to 1$ the space $H^1_{\tiny{\mbox{loc}}}$, is the following result
\begin{Proposition}\cite[Corollary 7]{bourgain:hal-00747692}\label{Brezis}
Let $\Omega\subset \R^n$ be a bounded domain. For $1 < p < \infty$, let $f_s\in W^{s,p}(\Omega)$, and assume that
$$
[f_s]_{W^{s,p}(\Omega)}\leq C_0.
$$
Then, up to a subsequence, $(f_s)$ converges in $L^p(\Omega)$ as $s\to 1$(and, in fact, in $W^{t,p}(\Omega)$, for all $t<1$) to some $f\in W^{1,p}(\Omega)$.
\end{Proposition}
In \cite{bourgain:hal-00747692} the authors used a different notation since in our paper the normalization constant $C(n,s)$ is incorporate in the seminorm $[\cdot]_{H^s}$, in order to obtain a continuity of the norm $\norm{\cdot}{H^s}$ for $s \in (0,1]$.

\subsection{Proof of Theorem \ref{teolimit1}}
\begin{proof}[\unskip\nopunct]
 Let $C$ be an open cone and $C_R$ be a regular cone with section on $S^{n-1}$ of class $C^{1,1}$ such that $C_R\subset C$ and $\partial C_R \cap \partial C = \{0\}$. 
\\\\By monotonicity of the homogeneity degree $\gamma_s(\cdot)$ with respect to the inclusion, we directly obtain $\gamma_s(C)< \gamma_s(C_R)$ and consequently, up to consider a subsequence, we obtain the existence of the following finite limits
\begin{equation}\label{upto}
  \overline{\gamma}(C) = \lim_{s \to 1}\gamma_{s}(C), \quad \mu(C) = \lim_{s \to 1}\frac{C(n,s)}{2s - \gamma_{s}(C)}.
\end{equation}
Since $\gamma_s(C)<2s$, then $\overline{\gamma}(C) \leq 2$ and similarly $\mu(C) \in [0,+\infty)$.\vspace{0.2cm}\\
Let $K\subset \R^n$ be a compact set and consider $x_0 \in K$ and $R>0$ such that $K \subset B_R(x_0)$. Given $\eta\in C^\infty_c(B_2)$, a radial cut off function such that $\eta\equiv 1$ in $B_1$ and $0\leq\eta\leq1$ in $B_2$, consider the rescaled function $\eta_K(x)= \eta(\frac{x-x_0}{R})$ which satisfies $\eta_K \equiv 1$ on $K$.\\
By Proposition \ref{propHs}, we have
$$
[u_s \eta_K]^2_{H^s(B_{2R}(x_0))}\leq [u_s\eta_K]^2_{H^s(\R^n)} \leq M(n,K)\left [\frac{C(n,s)}{2(1-s)} + \frac{C(n,s)}{2s - \gamma_s} \right],
$$
and similarly
\begin{align*}
\norm{u_s \eta_K}{H^s(B_{2R}(x_0))}^2 &\leq \norm{u_s \eta_K}{L^2(\R^n)}^2 + [u_s \eta_K]_{H^s(\R^n)}^2 \\
&\leq M(n,K)\left [\frac{C(n,s)}{2(1-s)} + \frac{C(n,s)}{2s - \gamma_s} +1\right]\\
& \leq M(n,K)\left [\frac{2n}{\omega_{n-1}} + c\mu(C) +1\right].
\end{align*}
By applying Proposition \ref{Brezis} with $\Omega= B_{2R}(x_0)$, we obtain that, up to a subsequence, $u_{s}\eta_K \to \overline{u}\eta_K$ in $L^2(B_{2R}(x_0))$ and
$$
\norm{\overline{u}\eta_K}{H^1(B_{2R}(x_0))}^2 \leq M(n,K)
$$
up to relabeling the constant $M(n,K)$.\\
By construction, since $\eta_K \equiv 1$ on $K$ and $\eta_K \in [0,1]$, we obtain that $u_s \to \overline{u}$ in $L^2(K)$ and similarly
$$
\norm{\overline{u}}{H^1(K)}\leq \norm{\overline{u} \eta_K}{H^1(K)}\leq \norm{\overline{u} \eta_K}{H^1(B_{2R}(x_0))}<\infty,
$$
which gives us the local integrability in $H^1(\R^n)$.\vspace{0.3cm}\\
By Proposition \ref{propcomp} and Corollary \ref{corollmu} we obtain, up to pass to a subsequence, uniform in $s$ bound in $C^{0,\alpha}_{\mathrm{loc}}(C)$ for $(u_{s})$. Then, since we obtain uniform convergence on compact subsets of $C$, the limit must be necessary nontrivial with $||\overline u||_{L^\infty(S^{n-1})}=1$, nonnegative and $\overline\gamma(C)$-homogeneous.\\


Let $\varphi \in C^{\infty}_c(C)$ be a positive smooth function compactly supported such that $\mbox{supp }\varphi\subset B_\rho$, for some $\rho>0$. By definition of the distributional fractional Laplacian
$$
0 = \int_{\R^n}{\varphi (-\Delta)^{s}u_{s} \mathrm{d}x} =  \int_{\R^n}{u_{s} (-\Delta)^{s}\varphi \mathrm{d}x} =\int_{\R^n\setminus B_{\rho}}{u_{s} (-\Delta)^{s}\varphi \mathrm{d}x} + \int_{B_{\rho}}{u_{s} (-\Delta)^{s}\varphi \mathrm{d}x}.
$$
Since
$$
\frac{1}{\abs{x-y}^{n+2s}} = \frac{1}{\abs{x}^{n+2s}}\left(1- (n+2s)\frac{y}{\abs{x}}  \int_{0}^1{\frac{\frac{x}{\abs{x}}-t\frac{y}{\abs{x}} }{\abs{\frac{x}{\abs{x}}-\frac{ty}{\abs{x}}}^{n+2s+2}}\mathrm{d}t}\right),
$$
by definition of the fractional Laplacian for regular functions, it follows
\begin{align*}
\int_{\R^n\setminus B_\rho}{u_s (-\Delta)^s \varphi \mathrm{d}x} =& C(n,s) \int_{\R^n \setminus B_{\rho}}{u_s(x) \int_{\mbox{supp }\varphi}{\frac{-\varphi(y)}{\abs{y-x}^{n+2s}}\mathrm{d}y}\mathrm{d}x}\\
=& C(n,s) \int_{\R^n\setminus B_\rho}{\frac{u_s(x)}{\abs{x}^{n+2s}}\int_{\mbox{supp }\varphi}{-\varphi(y)\mathrm{d}y}\mathrm{d}x}+\\
&+ C(n,s)(n+2s)\int_{\R^n \setminus B_\rho}{\frac{u_s(x)}{\abs{x}^{n+2s+1}}\psi(x)\mathrm{d}x},
\end{align*}
for some $\psi \in L^\infty$. Moreover, since $u_s$ is $\gamma_s(C)$-homogeneous with $\gamma_s(C) <2s$, we have
$$
C(n,s) \int_{\R^n\setminus B_\rho}{\frac{u_s(x)}{\abs{x}^{n+2s}}\mathrm{d}x} = \frac{C(n,s)}{2s-\gamma_s(C)} \rho^{\gamma_s(C) -2s}\int_{S^{n-1}}{u_s(\theta) \mathrm{d}\sigma}
$$
and similarly
$$
C(n,s) \abs{\int_{\R^n\setminus B_\rho}{\frac{u_s(x)}{\abs{x}^{n+2s+1}}\psi(x)\mathrm{d}x}} \leq \frac{C(n,s)\norm{\psi}{L^\infty}}{2s-\gamma_s(C)+1} \rho^{\gamma_s(C) -2s-1}\int_{S^{n-1}}{u_s(\theta) \mathrm{d}\sigma}=o(1).
$$
Hence, for each $s \in (0,1)$
\begin{align*}
\int_{B_{\rho}}{u_{s} (-\Delta)^{s}\varphi \mathrm{d}x} &=
\int_{\R^n\setminus B_{\rho}}{u_{s} (-\Delta)^{s}\varphi \mathrm{d}x}\\
&= C(n,s)\int_{\R^n\setminus B_{\rho}}{u_{s}(x)\int_{\mbox{supp }\varphi}{\frac{\varphi(y)}{\abs{x-y}^{n+2s}}\mathrm{d}y}\mathrm{d}x}\\
&= \frac{C(n,s)}{2s - \gamma_{s}(C)}\int_{\mbox{supp }\varphi}{\varphi(x)\mathrm{d}x}\int_{S_{n-1}}{u_{s}\mathrm{d}\sigma} + o(1)
\end{align*}
and passing through the limit, up to a subsequence, we obtain
\begin{align*}
\int_{B_\rho}{\overline{u}(-\Delta)\varphi\mathrm{d}x} &=  \mu(C)
\int_{S_{n-1}}{\overline{u}\mathrm{d}\sigma}
\int_{\mbox{supp }\varphi}{\varphi(x)\mathrm{d}x}\\
&=\int_{B_\rho}{\left(\mu(C)\int_{S_{n-1}}{\overline{u}\mathrm{d}\sigma} \right)\varphi(x)\mathrm{d}x},
\end{align*}
which implies, integrating by parts, that
$$
-\Delta \overline{u} = \mu(C)\int_{S^{n-1}}{\overline{u}\mathrm{d}\sigma} \quad \mbox{in }\mathcal{D}'(C).
$$
Since the function $\overline{u}$ is $\overline{\gamma}(C)$-homogenous, we get
\begin{equation}\label{limit}
-\Delta_{S^{n-1}}\overline{u} = \overline{\lambda}\overline{u} + \mu(C) \int_{S^{n-1}}{\overline{u}\mathrm{d}\sigma} \quad \mbox{on } S^{n-1}\cap C,
\end{equation}
where $\overline{\lambda} = \overline{\gamma}(C)(\overline{\gamma}(C)+n-2)$ is the eigenvalue associated to the critical exponent $\overline{\gamma}(C)\leq 2$. \\\\
Consider now a nonnegative $\varphi \in H^1_0(S^{n-1}\cap C)\setminus \{0\}$,  strictly positive on $S^{n-1}\cap C$ which achieves \eqref{lambda1}. Then
\begin{equation}\label{test1}
    -\Delta_{S^{n-1}}\varphi = \lambda_1(C) \varphi , \quad\mbox{in }\,\,H^{-1}(S^{n-1}\cap C).
\end{equation}
By testing this equation with $\overline{u}$ and integrating by parts, we obtain
\begin{equation}\label{compare}
\left( \lambda_1(C) - \overline{\lambda}\right) \int_{S^{n-1}}{\overline{u}\varphi \mathrm{d}\sigma} = \mu(C) \int_{S^{n-1}}{\overline{u}\mathrm{d}\sigma}\int_{S^{n-1}}{\varphi\mathrm{d}\sigma}\geq 0
\end{equation}
which implies that in general $\gamma(C) \geq \overline{\gamma}(C)$ and $\gamma(C)=\overline{\gamma}(C)$ if and only if $\mu(C)=0$.
\subsubsection{Wide cones: $\gamma(C)<2$}
By the previous remark we have $\overline{\gamma}(C)<2$ and by definition of $\mu(C)$, it follows $\mu(C)=0$. Since $\varphi$ is the trace on $S^{n-1}$ of an homogenous harmonic function on $C$,  we obtain that $\overline{\gamma}(C)=\gamma(C)$ and $\overline{u}$ is an homogeneous nonnegative harmonic function on $C$ such that $\norm{\overline{u}}{L^\infty(S^{n-1})}=1$.
\subsubsection{Narrow cones: $\gamma(C)\geq2$}
If $\overline{\gamma}(C)< 2$ we have $\mu(C)=0$ and consequently $\lambda_1(C)=\overline{\lambda}$, which is a contradiction since $\gamma(C)\geq 2>\overline{\gamma}(C)$. Hence, if $C$ is a narrow cone we get $\overline{\gamma}(C)=2$. Since $\gamma(C)=2$ is trivial and it follows directly from the previous computations, consider now $\mu_0(C)$ as the minimum defined in \eqref{mu01}, which is well defined and strictly positive since we are focusing on the remaining case $\gamma(C)>2$.
We already remarked that it is achieved by a nonnegative $\psi\in H^1_0(S^{n-1}\cap C)\setminus\{0\}$ which is strictly positive on $S^{n-1}\cap C$ and solution of
\begin{equation*}
-\Delta_{S^{n-1}} \psi=2n\psi+\mu_0(C)\displaystyle\int_{S^{n-1}}\!\!\psi\mathrm{d}\sigma \quad \mathrm{in }\,\,H^{-1}(S^{n-1}\cap C).
\end{equation*}
As we already did in the previous cases, by testing this equation with $\overline{u}$ we obtain $\mu(C)= \mu_0(C)$.\\
By uniqueness of the limits $\overline{\gamma}(C)$ and $\mu(C)$, the result in \eqref{upto} holds for $s\to 1$ and not just up to a subsequence.
\end{proof}
\begin{remark}\label{rem:uniqueness}
The possible obstruction to the existence  of the limit of $u_s$ as $s$ converge s to one lies in the possible lack of uniqueness of  nonnegative solutions to \eqref{limit.equation1} such that $\norm{\overline{u}}{L^\infty(S^{n-1})}=1$. This is the reason why we need to extract subsequences in the asymptotic analysis of Theorem  \ref{teolimit1}.    More precisely, uniqueness of \eqref{lambda1} implies uniqueness of the limit $\overline{u}$ in the case $\gamma(C)\leq 2$ and uniqueness of \eqref{mu01} in the case $\gamma(C)>2$. When $C$ is connected \eqref{lambda1} is attained by a unique normalized nonnegative solution via a standard argument based upon the maximum priciple.  On the other hand,  as we already remarked, when $\gamma(C)>2$, problem \eqref{mu01} always admits  a unique solution. Ultimately, the main obstacle in this analysis is the disconnection of the cone $C$ when $\gamma(C)\leq2$: in this case we cannot always ensure the uniqueness of the solution of the limit problem and even the positivity of the limit function $\overline{u}$ on every connected components of $C$.
\end{remark}
The following example  shows  uniqueness of the limit function $\overline{u}$ due to the nonlocal nature of the fractional Laplacian under a symmetry assumption on the cone $C$.
\begin{Proposition}
Let $C= C_1 \cup\dots\cup C_m$ be a union of disconnected cones such that  $C_1$ is connected and there are orthogonal  maps $\Phi_2,\dots,\Phi_m\in O(n) $ (e.g. reflections about hyperplanes) such that $C_i=\Phi_i(C_1)$ and and $\Phi_i(C)=(C)$ for $i=2,\dots,m$. Let $(u_s)$ be the family of nonnegative solutions to \eqref{Pstheta} such that $\norm{u_s}{L^\infty(S^{n-1})}=1$. Then there exists the limit of $u_s$ as $s\nearrow 1$  in $L^2_{\mathrm{loc}}(\mathbb{R}^n)$ and uniformly on compact subsets of $C$.
\end{Proposition}

\begin{proof}
We remark that, for any element of the orthogonal group $\Phi \colon \R^n \to \R^n$,
  \begin{align*}
    (-\Delta)^s \left(u\circ \Phi\right) (x)  =  C(n,s) \mbox{ P.V.}\int_{\R^n}{\frac{u(\Phi(x))-u(y)}{\abs{\Phi(x)-y}^{n+2s}}\mathrm{d}y}
     = (-\Delta)^s u\left(\Phi(x)\right)\;.
   \end{align*}
By the uniqueness result \cite[Theorem 3.2 ]{banuelos} of $s$-harmonic functions on cones, we infer that $u_s \equiv u_s\circ\Phi_i$, for every $i=2,\dots,m$. Therefore, there holds convergence to  $\overline{u}$, where
satisfies $\norm{\overline{u}}{L^\infty(S^{n-1})}=1$, and it is a solution of
\begin{equation}
\begin{cases}\label{limit.equation.sym}
-\Delta\overline u=\mu(C)\displaystyle\int_{S^{n-1}}{\!\!\!\!\overline u \mathrm{d}\sigma}& \mathrm{in} \ C,\\
\overline u \geq0 & \mathrm{in} \ C,\\
\overline u=0 & \mathrm{in} \ \mathbb{R}^n\setminus C\;,
\end{cases}
\end{equation}
 such that $\overline u\equiv \overline u\circ\Phi_i$ for every $i=2,\dots,m$. Finally, connectedness of $C_1$ yields uniqueness of such solution
also for narrow cones.
\end{proof}

\subsection{Proof of Corollary \ref{cor1}}
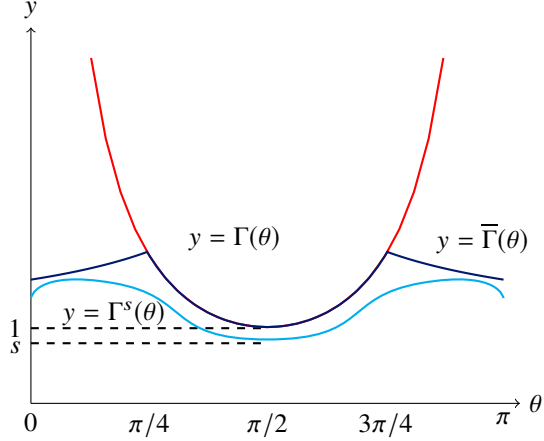
\begin{figure}[ht!]
\begin{tikzpicture}

          \draw[->] (0,0) -- (6.5,0) node[right] {$\theta$};
          \draw[->] (0,0) -- (0,5) node[above] {$y$};
          \draw	(0,0) node[anchor=north] {0}
		(pi/2,0) node[anchor=north] {$\pi/4$}
		(pi,0) node[anchor=north] {$\pi/2$}
		(3*pi/2,0) node[anchor=north] {$3\pi/4$}
		(2*pi,0) node[anchor=north] {$\pi$}
		(0,1) node[anchor=east] {$1$}
		(0,0.75) node[anchor=east] {$s$};
		
	\draw(2.7,2.2) node{$y=\Gamma(\theta)$}
		(1.1,1.2) node{$y=\Gamma^s(\theta)$}
        (6,2.2) node{$y=\overline{\Gamma}(\theta)$};

                   \draw [domain=0:pi,thick, dashed] plot (\x,1);
          \draw [domain=0:pi,thick,dashed] plot (\x,0.8);
          \draw [domain=0.8:2*pi-0.8,red,thick] plot (\x,{-1.85+4.5/\x+4.5/(2*pi-\x)});
         \draw [cyan,thick] plot [smooth, tension=1.2] coordinates { (0,1.4) (1.1,1.6) (pi,0.85) (2*pi-1.1,1.6)  (2*pi,1.4)};
           \draw [domain=pi/2-0.036:2*pi-pi/2+0.036,navyblue,thick=2] plot (\x,{-1.85+4.5/\x+4.5/(2*pi-\x)});
          \draw [domain=0:pi/2-0.036,navyblue,thick=2] plot [smooth] (\x,{0.93+4.5/(2*pi-1.38*\x)});
          \draw [domain=2*pi-pi/2+0.036:2*pi,navyblue,thick=2] plot (\x,{0.93+4.5/(-0.76*pi +1.38*\x)});

         \end{tikzpicture}
         \caption{Values of the limit $\overline{\Gamma}(\theta)=\lim_{s \to 1}\Gamma^s(\theta)$ and $\Gamma(\theta)$, for $n=2$.}
         \end{figure}

\begin{proof}[\unskip\nopunct]
Corollary \ref{cor1} is an easy application of our main Theorem \ref{teolimit1}, since it is a consequence of the Dini's Theorem for a monotone sequence of continuous functions which converges pointwisely to a continuous function on a compact set. In fact, fixed $s\in(0,1)$, the function $\theta\mapsto\gamma_s(\theta)$ is continuous in $[0,\pi)$ with $\gamma_s(0)=2s$ and $\gamma_s(\pi)=0$. Moreover this function is also monotone decreasing in $[0,\pi]$ and since there exists the limit
\begin{equation*}
\lim_{\theta\to\pi^-}\gamma_s(\theta)=\begin{cases}
\frac{2s-1}{2} &\mathrm{if \ }n=2 \ \mathrm{and \ }s>\frac{1}{2},\\
\gamma_s(\pi)=0 & \mathrm{otherwise},
\end{cases}
\end{equation*}
we can extend $\theta\mapsto\gamma_s(\theta)$ to a continuous function in $[0,\pi]$ (see \cite{MR2213639}). Nevertheless, the limit $\overline\gamma(\theta)=\lim_{s\to1}\gamma_s(\theta)=\min\{\gamma(\theta),2\}$ is continuous on $[0,\pi]$ with
\begin{equation*}
\overline\gamma(\pi)=\begin{cases}
\frac{1}{2} &\mathrm{if \ }n=2,\\
0 & \mathrm{otherwise}.
\end{cases}
\end{equation*}
Eventually, for any fixed $\theta\in[0,\pi]$, the function $s\mapsto\gamma_s(\theta)$ is monotone nondecreasing in $(0,1)$. By the Dini's Theorem the convergence is uniform on $[0,\pi]$. This fact obviously implies the uniform convergence
$$\Gamma^s(\theta)=\frac{\gamma_s(\theta)+\gamma_s(\pi-\theta)}{2}\longrightarrow\overline\Gamma(\theta)=\frac{\overline\gamma(\theta)+\overline\gamma(\pi-\theta)}{2}$$
in $[0,\pi]$, and hence
$$\nu_s^{ACF}=\min_{\theta\in[0,\pi]}\Gamma^s(\theta)\longrightarrow\min_{\theta\in[0,\pi]}\overline\Gamma(\theta)=\nu^{ACF}.$$
\end{proof}

\section{Uniform in $s$ estimates in $C^{0,\alpha}$ on annuli}
We have already remarked in Section 2 that, if you take a cone $C=C_\omega$ with $\omega\subset S^{n-1}$ a finite union of connected $C^{1,1}$ domain $\omega_i$, such that $\overline{\omega}_i \cup\overline{ \omega}_j = \emptyset$ for $i\neq j$, by \cite[Lemma 3.3]{MR2213639} we have \eqref{c(s)original}.\\ Hence
solutions $u_s$ to \eqref{Pstheta} are $C^{0,s}(S^{n-1})$ and for any fixed $\alpha\in(0,1)$, any solution $u_s$ with $s\in(\alpha,1)$ is $C^{0,\alpha}(S^{n-1})$; that is, there exists $L_s>0$ such that
\begin{equation*}
\sup_{x,y\in S^{n-1}}\frac{|u_s(x)-u_s(y)|}{|x-y|^\alpha}=L_s.
\end{equation*}
Let us consider an annulus $A=A_{r_1,r_2}=B_{r_2}\setminus \overline{B_{r_1}}$ with $0<r_1<r_2<+\infty$. We have the following result.
\begin{Lemma}\label{lemmanello}
Let $\alpha\in(0,1)$, $s_0\in(\max\{1/2,\alpha\},1)$ and $A$ an annulus centered at zero. Then there exists a constant $c>0$ such that any solution $u_s$ to \eqref{Pstheta} with $s\in[s_0,1)$ satisfies
\begin{equation*}
\sup_{x,y\in A}\frac{|u_s(x)-u_s(y)|}{|x-y|^\alpha}\leq cL_s.
\end{equation*}
\end{Lemma}
\proof
First of all we remark that
\begin{equation}\label{alphasph}
\sup_{x,y\in S^{n-1}_r}\frac{|u_s(x)-u_s(y)|}{|x-y|^\alpha}\leq cL_s,
\end{equation}
for any $r\in(r_1,r_2)$. In fact, by the $\gamma_s$-homogeneity of our solutions, we have
\begin{equation*}
\sup_{x,y\in S^{n-1}_r}\frac{|u_s(x)-u_s(y)|}{|x-y|^\alpha}= L_sr^{\gamma_s-\alpha},
\end{equation*}
and since $(2s_0-1)/2\leq\gamma_s(C)<2$ for any $s\in[s_0,1)$ by the inclusion $C\subset\mathbb{R}^n\setminus\{\mathrm{half-line \ from \ }0\}$, we obtain \eqref{alphasph}.\\\\
Now we can show what happens considering $x,y\in A$ which are not on the same sphere. We can suppose without loss of generality that $x\in S^{n-1}_{R}$, $y\in S^{n-1}_{r}$ with $r_1<r<R<r_2$. Hence let us take the point $z$ obtained by the intersection between $S^{n-1}_{r}$ and the half-line connecting $0$ and $x$ ($z$ may be $y$ itself). Hence
\begin{eqnarray*}
|u_s(x)-u_s(y)|&\leq&|u_s(x)-u_s(z)|+|u_s(z)-u_s(y)|\nonumber\\
&\leq&u_s(x/|x|)||x|^{\gamma_s}-|z|^{\gamma_s}|+cL_s|z-y|^\alpha\nonumber\\
&\leq&cL_s|x-y|^\alpha.
\end{eqnarray*}
In fact we remark that $||u_s||_{L^\infty(S^{n-1})}=1$. Moreover, since the angle $\beta=\widehat{xzy}\in(\pi/2,\pi]$, obviously $|z-y|^\alpha\leq|x-y|^\alpha$. Moreover by the $\alpha$-H\"older continuity of $t\mapsto t^{\gamma_s}$ in $(r_1,r_2)$ and the bounds $(2s_0-1)/2\leq\gamma_s(C)<2$, one can find a universal constant $c>0$ such that\begin{equation*}
||x|^{\gamma_s}-|z|^{\gamma_s}|\leq c||x|-|z||^\alpha\leq c|x-z|^\alpha\leq c|x-y|^\alpha,
\end{equation*}
where the last inequality holds since $z$ is the point on $S^{n-1}_{r}$ which minimizes the distance $\mbox{dist}(x,S^{n-1}_{r})$.
\endproof
\subsection{Proof of Theorem \ref{thm:holder}.}
\begin{proof}[\unskip\nopunct]
Seeking a contradiction,
\begin{equation}\label{contra}
\max_{x,y\in S^{n-1}}\frac{|u_{s_k}(x)-u_{s_k}(y)|}{|x-y|^\alpha}=L_{s_k}=L_k\to+\infty,\qquad\mathrm{as} \ s_k\to1.
\end{equation}
We can consider the sequence of points $x_k,y_k\in S^{n-1}$ which realizes $L_k$ at any step. It is easy to see that this couple belongs to $\overline{C}\cap S^{n-1}$. Moreover we can always think $x_k$ as the one closer to the boundary $\partial C\cap S^{n-1}$. Therefore, to have \eqref{contra}, we have $r_k=|x_k-y_k|\to0$. Hence, without loss of generality, we can assume that $x_k,y_k$ belong defenetively to the same connected component of $C$ and
\begin{equation*}
\frac{|u_{s_k}(y_k)-u_{s_k}(x_k)|}{r_k^\alpha}=L_k,\qquad\frac{y_k-x_k}{r_k}\to e_1.
\end{equation*}
Let us define
\begin{equation*}
u^k(x)=\frac{u_{s_k}(x_k+r_kx)-u_{s_k}(x_k)}{r_k^\alpha L_k},\qquad x\in\Omega_k=\frac{C-x_k}{r_k}.
\end{equation*}
We remark that $u^k(0)=0$ and $u^k((y_k-x_k)/r_k)=1$.\\\\
Moreover we can have two different situations.
\begin{itemize}
\item[$\mathbf{Case \ 1:}$] If
$$\frac{r_k}{\mbox{dist}(x_k,\partial C)}\to0,$$
then the limit of $\Omega_k$ is $\mathbb{R}^n$.
\item[$\mathbf{Case \ 2:}$] If
$$\frac{r_k}{\mbox{dist}(x_k,\partial C)}\to l\in(0,+\infty],$$
then the limit of $\Omega_k$ is an half-space $\mathbb{R}^n\cap\{x_1>0\}$.
\end{itemize}
In any case let us define $\Omega_\infty$ this limit set. Let us consider the annulus $A^*:=B_{3/2}\setminus\overline {B_{1/2}}$. By Lemma \ref{lemmanello} and the definition of $u^k$, we obtain, for any $k$,
\begin{equation}
\sup_{x,y\in A^*_k}\frac{|u^k(x)-u^k(y)|}{|x-y|^\alpha}\leq c,
\end{equation}
where $A^*_k:=\frac{A^*-x_k}{r_k}\to\mathbb{R}^n$ and the constant $c>0$ depends only on $\alpha$ and $A^*$. Let us consider a compact subset $K$ of $\Omega_\infty$. Since for $k$ large enough $K\subset A^*_k$, functions $u^k$ are $C^{0,\alpha}(K)$ uniformly in $k$. This is due also to the fact that they are uniformly in $L^\infty(K)$, since $|u^k(x)-u^k(0)|\leq c|x|^\alpha$ on $K$. Hence $u^k\to\overline u$ uniformly on compact subsets of $\Omega_\infty$. Moreover $\overline u$ is globally $\alpha$-H\"older continuous and it is not constant, since $\overline u(e_1)-\overline u(0)=1$. To conclude, we will show that $\overline u$ is harmonic in the limit domain $\Omega_\infty$; that is, for any $\phi\in C^\infty_c(\Omega_\infty)$
\begin{equation*}
\int_{\Omega_\infty}\phi(-\Delta)\overline u \mathrm{d}x=0,
\end{equation*}
and this fact will be a contradiction with the global H\"older continuity. In fact we can apply Corollary 2.3 in \cite{MR2599456}, if  $\Omega_\infty=\mathbb{R}^n$ directly on the function $\overline u$ and if $\Omega_\infty=\mathbb{R}^n\cap\{x_1>0\}$, since $\overline u=0$ in $\partial\Omega_\infty$, we can use the same result over its odd reflection. Hence we want to prove
\begin{equation*}
\int_{\Omega_\infty}\phi(-\Delta)\overline u\mathrm{d}x=\int_{\Omega_\infty}\overline u(-\Delta)\phi \mathrm{d}x=\lim_{k\to+\infty}\int_{B_R}u^k(-\Delta)^{s_k}\phi \mathrm{d}x=0,
\end{equation*}
where $B_R$ contains the support of $\phi$ and the second equality holds by the uniform convergences  $u^k\to\overline u$ and $(-\Delta)^{s_k}\phi\to(-\Delta)\phi$ on compact subsets of $\Omega_\infty$, since $\phi$ is a smooth function compactly supported. Moreover, since $u^k$ is $s_k$-harmonic on $\Omega_k$, and for $k$ large enough the support of $\phi$ is contained in this domain, we have
\begin{equation*}
\int_{\mathbb{R}^n}u^k(-\Delta)^{s_k}\phi\mathrm{d}x=\int_{\mathbb{R}^n}{\phi(-\Delta)^{s_k}u^k \mathrm{d}x}=0.
\end{equation*}
In order to conclude we want
\begin{equation*}
\lim_{k\to+\infty}\int_{\mathbb{R}^n\setminus B_R}u^k(-\Delta)^{s_k}\phi \mathrm{d}x=0.
\end{equation*}
Hence, defining $\eta=x_k+r_kx$ and using Remark \ref{labogdancns}, we obtain
\begin{equation*}
\left|\int_{\mathbb{R}^n\setminus B_R}u^k(-\Delta)^{s_k}\phi\mathrm{d}x\right|\leq\frac{C(n,s_k)}{L_k}r_k^{2s_k-\alpha}\int_{|\eta-x_k|>Rr_k}\frac{|u_{s_k}(\eta)-u_{s_k}(x_k)|}{|\eta-x_k|^{n+2s_k}}\mathrm{d}\eta.
\end{equation*}
For $k$ large enough, we notice that we can choose $\varepsilon>0$ such that the set $\{\eta\in \R^n \ : \ Rr_k<|\eta-x_k|<\varepsilon\}$ is contained in $A^*$. So, we can split the integral obtaining\vspace{0.3cm}
$$
\vspace{0.3cm}\int_{|\eta-x_k|>Rr_k}\!\!\frac{|u_{s_k}(\eta)-u_{s_k}(x_k)|}{|\eta-x_k|^{n+2s_k}}\mathrm{d}\eta \leq \int_{R r_k<|\eta-x_k|<\varepsilon}\!\!\frac{|u_{s_k}(\eta)-u_{s_k}(x_k)|}{|\eta-x_k|^{n+2s_k}}\mathrm{d}\eta +\int_{|\eta-x_k|>\varepsilon}\!\!\frac{|u_{s_k}(\eta)-u_{s_k}(x_k)|}{|\eta-x_k|^{n+2s_k}}\mathrm{d}\eta $$
where we have\vspace{0.2cm}
\begin{align*}
\frac{C(n,s_k)r_k^{2s_k-\alpha}}{L_k}\int_{R r_k<|\eta-x_k|<\varepsilon}\!\!\frac{|u_{s_k}(\eta)-u_{s_k}(x_k)|}{|\eta-x_k|^{n+2s_k}}\mathrm{d}\eta &\leq C(n,s_k)r_k^{2s_k-\alpha}c\omega_{n-1}\int_{Rr_k}^{\varepsilon}t^{-1+\alpha-2s_k}\mathrm{d}t\\
=&\frac{C(n,s_k)c\omega_{n-1}}{2s_k-\alpha}\left(R^{\alpha-2s_k}-\frac{r^{2s_k-\alpha}_k}{\varepsilon^{2s_k-\alpha}}\right)
\end{align*}
and similarly
\begin{align*}
\frac{C(n,s_k)r_k^{2s_k-\alpha}}{L_k}\int_{|\eta-x_k|>\varepsilon}\!\!\frac{|u_{s_k}(\eta)-u_{s_k}(x_k)|}{|\eta-x_k|^{n+2s_k}}\mathrm{d}\eta &\leq \frac{C(n,s_k)r_k^{2s_k-\alpha}c\omega_{n-1}}{L_k}\int_{\varepsilon}^\infty \frac{(1+t)^{\gamma_{s_k}}}{{t}^{1+2s_k}}\mathrm{d}t\\
=&\frac{C(n,s_k)r_k^{2s_k-\alpha}c\omega_{n-1}}{L_k}\left(1+\frac{\varepsilon^{\gamma_{s_k}-2s_k}}{2s_k-\gamma_{s_k}}\right).
\end{align*}
Finally, recalling that $r_k\to0$, $C(n,s_k)\to0$, $L_k\to \infty$ and  $2s_k-\alpha>0$ taking $s_0>1/2$, we obtain
$$
\left|\int_{\mathbb{R}^n\setminus B_R}u^k(-\Delta)^{s_k}\phi\mathrm{d}x\right|\leq \left(C(n,s_k)+ \frac{C(n,s_k)}{2s_k-\gamma_{s_k}}\frac{r_n^{2s_k-\alpha}}{L_k}\right)M
$$
which converges to zero as we claimed, since
$$\frac{C(n,s_k)}{2s_k-\gamma_{s_k}(C)}\to\mu(C)\in[0,+\infty)$$
in any regular cone $C\subset \R^n$.
\end{proof}

\bibliography{Biblio}
\bibliographystyle{abbrv}

\end{document}